\documentclass[12pt]{article}

\usepackage{amsmath,amsfonts,amsthm,amscd}
\usepackage{latexsym}

\newtheorem{thm}{Theorem}[section]
\newtheorem{pro}[thm]{Proposition}
\newtheorem{cor}[thm]{Corollary}
\newtheorem{lem}[thm]{Lemma}

\newtheorem{rem}[thm]{Remark}

\newcommand{\ltimes}{\: \rhd \! \! \! <}

\newcommand{\noin}{\noindent}

\newcommand{\dozspace}{\;\;\;\;\;\;\;\;\;\;\;\;}

\newcommand{\CC}{{\mathbb{C}}}

\newcommand{\HH}{{\mathbb{H}}}
\newcommand{\II}{{\mathbb{I}}}
\newcommand{\JJ}{{\mathbb{J}}}
\newcommand{\KK}{{\mathbb{K}}}
\newcommand{\LL}{{\mathbb{L}}}
\newcommand{\MM}{{\mathbb{M}}}

\newcommand{\QQ}{{\mathbb{Q}}}
\newcommand{\RR}{{\mathbb{R}}}

\newcommand{\VV}{{\mathbb{V}}}
\newcommand{\ZZ}{{\mathbb{Z}}}

\newcommand{\openvec}{\left( \!\! \begin{array}{c}}
\newcommand{\closevec}{\end{array} \!\! \right)}
\newcommand{\openmat}{\left( \!\! \begin{array}{rr}}
\newcommand{\closemat}{\end{array} \!\! \right)}
\newcommand{\opentri}{\left( \!\! \begin{array}{rrr}}
\newcommand{\closetri}{\end{array} \!\! \right)}
\newcommand{\openquad}{\left( \!\! \begin{array}{rrrr}}
\newcommand{\closequad}{\end{array} \!\! \right)}
\newcommand{\openquin}{\left( \!\! \begin{array}{rrrrr}}
\newcommand{\closequin}{\end{array} \!\! \right)}

\newcommand{\la}{\langle}
\newcommand{\ra}{\rangle}

\newcommand{\dash}{\mbox{\rm{---}}}

\newcommand{\DD}{\mbox{\rm{DD}}}

\newcommand{\GL}{\mbox{\rm{GL}}}
\newcommand{\SD}{\mbox{\rm{SD}}}

\newcommand{\Aut}{\mbox{\rm{Aut}}}

\newcommand{\Die}{\mbox{\rm{Die}}}
\newcommand{\End}{\mbox{\rm{End}}}

\newcommand{\Fein}{\mbox{\rm{Fein}}}
\newcommand{\Fix}{\mbox{\rm{Fix}}}
\newcommand{\Gal}{\mbox{\rm{Gal}}}

\renewcommand{\Im}{\mbox{\rm{Im}}}

\newcommand{\Ker}{\mbox{\rm{Ker}}}

\newcommand{\Mod}{\mbox{\rm{Mod}}}

\newcommand{\Typ}{\mbox{\rm{Typ}}}
\newcommand{\Vtx}{\mbox{\rm{Vtx}}}

\newcommand{\cs}{\mbox{\rm{cs}}}
\newcommand{\ddef}{\mbox{\rm{def}}}
\newcommand{\die}{\mbox{\rm{die}}}
\newcommand{\ex}{\mbox{\rm{ex}}}
\newcommand{\fs}{\mbox{\rm{fs}}}

\newcommand{\ind}{\mbox{\rm{ind}}}
\renewcommand{\inf}{\mbox{\rm{inf}}}
\newcommand{\lin}{\mbox{\rm{lin}}}
\renewcommand{\mod}{\mbox{\rm{mod}}}
\newcommand{\ppar}{\mbox{\rm{par}}}
\newcommand{\res}{\mbox{\rm{res}}}
\newcommand{\rip}{\mbox{\rm{rip}}}
\newcommand{\rk}{\mbox{\rm{rk}}}

\newcommand{\sn}{\mbox{\rm{sn}}}

\newcommand{\torn}{\mbox{\rm{torn}}}
\newcommand{\tr}{\mbox{\rm{tr}}}

\newcommand{\cD}{{\cal D}}
\newcommand{\cE}{{\cal E}}

\newcommand{\cH}{{\cal H}}
\newcommand{\cI}{{\cal I}}

\newcommand{\oopsi}{{\overline{\psi}}}
\newcommand{\ooxi}{{\overline{\xi}}}
\newcommand{\ooO}{{\overline{O}}}
\newcommand{\ooR}{{\overline{R}}}
\newcommand{\oodie}{\overline{\mbox{\rm{die}}} \, }
\newcommand{\ooDie}{\overline{\mbox{\rm{Die}}} \, }

\begin{document}

\title{Genotypes of irreducible \\
representations of finite $p$-groups}

\author{Laurence Barker}
\maketitle

\small

\begin{center}
Department of Mathematics, Bilkent University,
06800 Bilkent, Ankara, Turkey. \\
{\it e-mail:} barker@fen.bilkent.edu.tr.
\end{center}

\begin{abstract}
\noin For any characteristic zero coefficient field,
an irreducible representation of a finite $p$-group
can be assigned a Roquette $p$-group, called the
genotype. This has already been done by Bouc and
Kronstein in the special cases $\QQ$ and $\CC$. A
genetic invariant of an irrep is invariant under
group isomorphism, change of coefficient field,
Galois conjugation, and under suitable inductions
from subquotients. It turns out that the genetic
invariants are precisely the invariants of the
genotype. We shall examine relationships between
some genetic invariants and the genotype. As an
application, we shall count Galois conjugacy
classes of certain kinds of irreps.

\smallskip
\noin 2000 {\it Mathematics Subject Classification.}
Primary: 20C15; Secondary: 19A22.

\smallskip
\noin {\it Keywords:} genetic subquotients, conjugacy
classes of irreducible representations, Burnside rings
of finite $2$-groups.
\end{abstract}

\section{Introduction and conclusions}

We shall be concerned with $\KK G$-irreps, that
is to say, irreducible representations of $G$ over
$\KK$, where $G$ is a finite $p$-group, $p$ is a
prime, and $\KK$ is a field with characteristic
zero. Of course, in the study of the irreps of a
finite $p$-group over a field, there is scant loss
of generality in assuming that the field has
characteristic zero. Roquette \cite{Roq58} showed
that every normal abelian subgroup of $G$ is cyclic
if and only if $G$ is one of the following groups:
the cyclic group $C_{p^m}$ with $m \geq 0$; the
quaternion group $Q_{2^m}$ with $m \geq 3$; the
dihedral group $D_{2^m}$ with $m \geq 4$; the
semidihedral group $\SD_{2^m}$ with $m \geq 4$.
When these two equivalent conditions hold, we
call $G$ a {\bf Roquette $p$-group}. This paper
is concerned with a reduction technique whereby
the study of $\KK G$-irreps reduces to the case
where $G$ is Roquette.

The reduction technique originates in Witt
\cite{Wit52} and Roquette \cite{Roq58}. Our main
sources are: Kronstein \cite{Kro66}, Iida--Yamanda
\cite{IY92} for complex irreps of $p$-groups; tom
Dieck \cite[Section III.5]{Die87} for real irreps
of finite nilpotent groups; Bouc \cite{Bou04},
\cite{Bou05}, \cite{Bou*} for rational irreps of
$p$-groups; Hambleton--Taylor--Williams \cite{HTW90},
Hambleton--Taylor \cite{HT99}, for rational irreps
of hyperelementary groups.

We shall be taking advantage of the generality
of our scenario. In the final section, we shall
unify some enumerative results of tom Dieck
\cite{Die87} and Bouc \cite{Bou*} concerning
Galois conjugacy classes of rational, real and
complex irreps.

Consider a $\KK G$-irrep $\psi$. In a moment, we
shall define a Roquette $p$-group $\Typ(\psi)$,
which we shall call the {\bf genotype} of $\psi$.
We shall explain how $\Typ(\psi)$ determines ---
and is determined by --- many other invariants of
$\psi$.

Let us agree on some terminology. When no ambiguity
can arise, we may neglect to distinguish between
characters, modules and representations. For a
$\KK G$-rep $\mu$, we write $\QQ[\mu]$ for the
field generated over $\QQ$ by the values of the
character $\mu$. We write $\End_{\KK G}(\mu)$ to
denote the endomorphism algebra of the of the
module $\mu$. We write $\Ker(\mu)$ to denote the
kernel of the representation $\mu$ as a group
homomorphism from $G$. When $\mu$ is irreducible,
the Wedderburn component of $\KK G$ associated
with $\mu$ is the Wedderburn component that is
not annihilated by the representation $\mu$ as
an algebra homomorphism from $\KK G$.

Given subgroups $K \unlhd H \leq G$, then
the subquotient $H/K$ of $G$ is said to be
{\bf strict} provided $H < G$ or $1 < K$.
We understand induction $\ind_{H/K}^G$ to be
the composite of induction $\ind_H^G$ preceded
by inflation $\inf_{H/K}^H$.
An easy application of Clifford theory shows that,
if some $\KK G$-irrep is not induced from a strict
subquotient, then $G$ is Roquette. Therefore,
any $\KK$-irrep of a finite $p$-group is induced
from a Roquette subquotient. For example, the
faithful $\CC D_8$-irrep $\psi_0$ is induced from
a faithful $\CC C_4$-irrep $\phi_0$. But this
observation, in its own, does not yield a very
powerful reduction technique. The $\CC$-irreps
$\psi_0$ and $\phi_0$ differ in some important
respects, for instance, $\QQ[\psi_0] = \QQ$
whereas $\QQ[\phi_0] = \QQ[i]$.

Since $\KK$ has characteristic zero, we can
equally well understand deflation
$\ddef_{H/K}^H$ to be passage to the $K$-fixed
points or as passage to the $K$-cofixed points.
We understand restriction $\res_{H/K}^H$ to
be the composite of deflation $\ddef_{H/K}^H$
preceded by restriction $\res_H^G$. A
$\KK G$-irrep $\psi$ is said to be {\bf tightly
induced} from a $\KK H/K$-irrep $\phi$ provided
$\psi = \ind_{H/K}^G(\phi)$ and no Galois
conjugate of $\phi$ occurs in the
$\KK H/K$-rep $\res_{H/K}^G(\psi) - \phi$.
This is equivalent to the condition that,
regarding $\phi$ as a $\KK H$-irrep by
inflation, then $\psi = \ind_H^G(\phi)$
and no Galois conjugate of $\phi$ occurs in
the $\KK H$-rep $\res_H^G(\psi) - \phi$.
So, when discussing tight induction, the
inflations and deflations are trivial
formalities, and we may safely regard
$\KK H/K$-reps as $\KK H$-reps by inflation.

\begin{thm}
{\rm (Genotype Theorem)} Given a $\KK G$-irrep
$\psi$, then there exists a Roquette subquotient
$H/K$ such that $\psi$ is tightly induced from a
faithful $\KK H/K$-irrep $\phi$. For any such
subquotient $H/K$, the $\KK H/K$-irrep $\phi$ is
unique. Given another such subquotient $H'/K'$,
then $H/K \cong H'/K'$.
\end{thm}

We call $H/K$ a {\bf genetic subquotient} for
$\psi$, and we call $\phi$ the {\bf germ} of
$\psi$ at $H/K$. We define the {\bf genotype}
of $\psi$, denoted $\Typ(\psi)$, to be $H/K$
regarded as an abstract group, well-defined
only up to isomorphism. The existence of such
subquotients $H/K$, in the case $\KK = \QQ$,
is implicit in Witt \cite{Wit52}, explicit in
Roquette \cite{Roq58}. The uniqueness, in the
case $\KK = \QQ$, is due to Bouc \cite{Bou04}.
Via Lemma 3.2, we see that the existence and
uniqueness, in the case $\KK = \CC$, is due
to Kronstein \cite{Kro66}.

In Section 4, we shall prove the Genotype Theorem
1.1 indirectly by invoking the Field-Changing
Theorem 3.5, which says that the genetic theory
is independent of the field $\KK$. As a matter of
fact, the theory really is independent of $\KK$,
and there is no need to reduce to a previously
established special case. A direct proof of the
Genotype Theorem will materialize from some
characterizations of $\Typ(\psi)$ in Section 5.

Given a $\KK G$-irrep $\psi$, then there exists
a unique $\QQ G$-irrep $\psi_\QQ$ such that $\psi$
occurs in the $\KK G$-rep $\KK \psi_\QQ = \KK
\otimes_\QQ \psi_\QQ$. For a field $\LL$ with
characteristic zero and an $\LL G$-irrep $\psi'$,
we say that $\psi$ and $\psi'$ are {\bf
quasiconjugate} provided $\psi_\QQ = \psi'_\QQ$.
We write $\psi_\LL$ to denote an arbitrarily
chosen $\LL G$-irrep that is quasiconjugate to
$\psi$. In Section 2, as a little illustrative
application of the genetic reduction technique,
we shall show that, for irreps of finite
$p$-groups over an arbitrary field with
characteristic zero, the notion of Galois
conjugacy is well-defined and well-behaved.
Corollary 2.6 says that two $\KK G$-irreps are
quasiconjugate if and only if they are Galois
conjugate.

Consider a formal invariant $\cI$ defined for
all irreps of all finite $p$-groups over all
fields with characteristic zero. We call $\cI$
a {\bf quasiconjugacy invariant} provided
$\cI(\psi) = \cI(\psi')$ for all characteristic
zero fields $\LL$ and all $\LL G$-irreps $\psi'$
that are quasiconjugate to $\psi$. If $\cI$ is
a quasiconjugacy invariant then, in particular,
it is a Galois conjugacy invariant, and
$\cI(\psi_\LL)$ is well-defined, independently
of the choice of $\psi_\LL$. We call $\cI$ a
{\bf global invariant} provided $I(\psi)
= I(\psi^\&)$ whenever some group isomorphism
$G \rightarrow G^\&$ sends the $\KK G$-irrep
$\psi$ of $G$ to the $\KK G^\&$-irrep $\psi^\&$
of $G^\&$. For instance, $\psi_\QQ$ is a
quasiconjugacy invariant but not a global
invariant, while the degree $\psi(1)$ a global
invariant but not a quasiconjugacy invariant.

We call $\cI$ a {\bf tight induction invariant}
provided $\cI(\psi) = \cI(\phi)$ for all
subquotients $H/K$ of $G$ and all $\KK H/K$-irreps
$\phi$ such that $\psi$ tightly induced from
$\phi$. We call $\cI$ a {\bf genetic invariant}
when $\cI$ is a tight quasiconjugacy global
invariant, in other words, $\cI$ is preserved by
tight induction, Galois conjugacy, change of
field, and group isomorphism.

Despite the apparent strength of the defining
conditions, many interesting invariants of $\psi$
are genetic invariants. See the list at the end
of Section 2. A $\CC G$-irrep that is quasiconjugate
to $\psi$ is called a {\bf vertex} of $\psi$. The
number of vertices, denoted $v(\psi)$, is called
the {\bf order} of $\psi$. In Section 5, we shall
see that that $v(\psi)$ is a genetic invariant.
Another genetic invariant is the set of vertices
$\Vtx(\psi)$, regarded as a permutation set for
a suitable Galois group. Yet another genetic
invariant is the {\bf vertex field} $\VV(\psi)$,
which is the field generated over $\QQ$ by the
character values of a vertex. We shall also see
that the genotype $\Typ(\psi)$ is a genetic
invariant. In fact, Corollary 5.9 asserts that
the genetic invariants of $\psi$ are precisely
the invariants of $\Typ(\psi)$.

How can $\Typ(\psi)$ be ascertained from easily
calculated genetic invariants such as the order
$v(\psi)$ and the vertex set $\Vtx(\psi)$ and
the vertex field $\VV(\psi)$? How can
$\Typ(\psi)$ be used to ascertain less tractable
genetic invariants such as the minimal
splitting fields? We shall respond to these
questions in Section 5. Employing a medical
analogy: the patient has red eyes and long
teeth, therefore the patient has genotype
$V_{666}$, and therefore the patient is
allergic to sunlight. Or, arguing from
information in the next paragraph: if
$v(\psi) = 2$ and the Frobenius--Schur
indicator of $\psi$ is positive, then $\psi$
has genotype $D_{16}$, hence the unique minimal
splitting field for $\psi$ is $\QQ[\sqrt{2}]$.

Some examples: the genotype $\Typ(\psi)$ is
the trivial group $C_1$ if and only if $\psi$
is the trivial character; $\Typ(\psi) = C_2$ if
and only if $\psi$ is affordable over $\QQ$ and
non-trivial; $\Typ(\psi) = D_{2^m}$ with $m \geq 4$
if and only if $\psi_\CC$ is affordable over
$\RR$ but not over $\QQ$, in which case $m$ is
determined by the order $v(\psi) = 2^{m-3}$.

\section{Galois conjugacy of irreps of $p$-groups}

One starting-point for the genetic theory is
the following weak expression of ideas in Witt
\cite{Wit52}, Roquette \cite{Roq58}. (Although,
as we shall explain at the end of this paper,
the starting point for the work was actually
Tornehave \cite{Tor84}.) We are working
with the finite $p$-group $G$ because we have
nothing of novel significance to say about arbitrary
finite groups. The remark can be quickly obtained by
ignoring most of the proof of Lemma 4.2 below.

\begin{rem}
Any $\KK$-irrep of $G$ can be expressed in the form
$\ind_{H/K}^G(\phi)$ where $K \unlhd H \leq G$ and
the subquotient $H/K$ is a Roquette $p$-group and
$\phi$ is a faithful $\KK H/K$-irrep which is not
induced from any proper subgroup of $H/K$.
\end{rem}

In the notation of the remark, the subquotient
$H/K$ need not be unique up to isomorphism. When
$\KK = \CC$ or $\KK = \RR$, examples of the
non-uniqueness of $H/K$ abound. When $\KK = \QQ$,
an example of the non-uniqueness of $H/K$ is
supplied by the group $C_4 * D_{16}$. Here, the
smash product identifies the two central subgroups
with order $2$. Incidently, the group $C_4 * D_{16}$
was exhibited by Bouc \cite[7.7]{Bou04} as a
counter-example to another assertion. Routine
calculations show that, for the unique faithful
$\QQ C_4 * D_{16}$-irrep, one choice of $H/K$
has the form $(C_4 \times C_2)/C_2 \cong C_4$
and another choice of $H/K$ has the form
$(C_8 \times C_2)/C_2 \cong C_8$.

Although the remark yields only a crude
version of the genetic reduction technique, we
shall be applying it, in this section, to prove
the following theorem. Since the theorem is
fundamental, classical in style and not very hard
to obtain, one presumes that it is well-known, but
the author has been unable to locate it in the
literature. (Incidently, the author does not know
whether it holds for all hyperelementary groups.
A negative or absent answer might present an
inconvenience to the generalization of the
genetic theory to hyperelementary groups.)

We throw some terminology. Given a $\KK G$-irrep
and a subfield $\JJ \leq \KK$, we define the
$\JJ G$-irrep {\bf containing} $\psi$ to be the
unique $\JJ G$-irrep $\psi'$ such that $\psi$
occurs in the $\KK$-linear extension $\KK \psi'$.

For a positive integer $n$, we write $\QQ_n$ to
denote the field generated over $\QQ$ by
primitive $n$-th roots of unity. We call $\QQ_n$
the {\bf cyclotomic field} for {\bf exponent $n$}.
By a {\bf subcyclotomic field}, we mean a subfield
of a cyclotomic field. Since the Galois group
$\Aut(\QQ_n) = \Gal(\QQ_n/\QQ)$ is abelian, any
subcyclotomic field is a Galois extension of
$\QQ$. Observe that, for any $\KK$-irrep $\psi$,
the field $\QQ[\psi]$ is subcyclotomic.

\begin{thm}
Let $\LL$ be an extension field of $\KK$, let
$\psi$ be a $\KK$-irrep, and let $\psi_1$, $...$,
$\psi_v$ be the $\LL G$-irreps contained in the
$\LL$-linear extension $\LL \psi = \LL
\otimes_\KK \psi$. Then:

\smallskip
\noin {\bf (1)} Given $j$, then $\psi$ is the
unique $\KK G$-irrep containing $\psi_j$.

\smallskip
\noin {\bf (2)} There exists a positive integer
$m_\KK^\LL(\psi)$ such that $\LL \psi =
m_\KK^\LL(\psi)(\psi_1 + ... + \psi_v)$.

\smallskip
\noin {\bf (3)} The field $\QQ[\psi_j] =
\QQ[\psi_1]$ is a Galois extension of $\QQ[\psi]$.

\smallskip
\noin {\bf (4)} $\Gal(\QQ[\psi_1]/\QQ[\psi])$
acts freely and transitively on $\psi_1$, $...$,
$\psi_v$. The action is such that an element
$\alpha$ of the Galois group sends $\psi_j$ to
$\psi_k$ when ${}^\alpha(\psi_j(g)) = \psi_k(g)$
for all $g \in G$.
\end{thm}

Part (1) is obvious. Part (2) is an immediate
implication of part (4). By a comment above,
the subcyclotomic field $\QQ[\psi_j]$ is a
Galois extension of the subcyclotomic field
$\QQ[\psi]$. The equality $\QQ[\psi_j] =
\QQ[\psi_1]$ is implied by part (4). When we
have proved part (4), we shall have proved
the whole theorem.

If the extension $\LL/\KK$ is Galois then, as
explained in Curtis--Reiner
\cite[7.18, 7.19]{CR87}, $\Gal(\LL/\KK)$ acts
transitively on $\psi_1$, $...$, $\psi_v$. Any
element of $\Gal(\LL/\KK)$ restricts to an
element of $\Gal(\QQ[\psi_1]/\QQ[\psi])$. A
straightforward argument now establishes the
theorem in the case where $\LL/\KK$ is Galois.

Replacing $\LL/\KK$ by $\KK/\QQ$, we see that
the theorem implies the following proposition.

\begin{pro}
Let $\psi$ be a $\KK G$-irrep, and let $\alpha$
be an automorphism of a field containing
$\QQ[\psi]$. Then there exists a $\KK G$-irrep
${}^\alpha \psi$ such that $({}^\alpha \psi)(g)
= {}^\alpha(\psi(g))$ for all $g \in G$.
\end{pro}

Let us show that, conversely, the proposition
implies the theorem. Assuming the proposition,
it is easy to deduce that
$\Gal(\QQ[\psi_1]/\QQ[\psi])$ acts freely on
$\psi_1$, $...$, $\psi_v$. It remains only to
show that the action is transitive. Let $\JJ$
be an extension field of $\LL$ such that
$\JJ/\KK$ is Galois. Let ${}_1 \psi$ and
${}_j \psi$ be $\JJ G$-irreps contained in
$\psi_1$ and $\psi_j$, respectively. Since the
theorem holds for the Galois extension $\JJ/\KK$,
there exists an element $\alpha \in \Gal(\JJ/\KK)$
such that ${}_1^\alpha \psi = {}_j \psi$. Then
${}^\alpha \psi_1 = \psi_j$ and ${}^\alpha \psi
= \psi$. Also $\alpha$ restricts to an element of
$\Gal(\QQ[\psi_1]/\QQ[\psi])$. The transitivity
of the action is established. We have deduced
part (4) of the theorem. In fact, we have shown
that the proposition and the theorem are
equivalent to each other.

By the remark, proof of the theorem and the
proposition reduces to the case where $G$ is
Roquette. We must recall the classification of
the Roquette $p$-groups. First, let us recall the
members of a slightly different class of extremal
$p$-groups. The following groups are precisely the
$p$-groups with a self-centralizing cyclic maximal
subgroup. See, for instance, Aschbacher
\cite[23.4]{Asc86}. For $m \geq 3$, the {\bf
modular group} with order $p^m$ is defined to be
$$\Mod_{p^m} = \la a, c : a^{p^{m-1}} = c^p = 1
  \, , \: cac^{-1} = a^{p^{m-2} + 1} \ra .$$
Still with $m \geq 3$, the {\bf quaternion group}
with order $2^m$ is
$$Q_{2^m} = \la a, x : a^{2^{m-1}} = 1 \, , \: x^2
  = a^{2^{m-2}} \, , \: x a x^{-1} = a^{-1} \ra .$$
Again with $m \geq 3$, the {\bf dihedral group}
with order $2^m$ is
$$D_{2^m} = \la a, b : a^{2^{m-1}} = b^2 = 1
  \, , \: bab^{-1} = a^{-1} \ra .$$
For $m \geq 4$, the {\bf semidihedral group}
with order $2^m$ is
$$\SD_{2^m} = \la a, d : a^{2^{m-1}} = d^2 = 1
  \, , \: dad^{-1} = a^{2^{m-2} - 1} \ra .$$
We shall refer to these presentations as the
{\bf standard presentations}. The only
coincidence in the list is $\Mod_8 \cong D_8$.
Where the presentations make sense for smaller
values of $m$, the resulting groups are abelian.

Suppose that $G$ is a non-abelian Roquette
$p$-group and let $A$ be a maximal normal cyclic
subgroup of $G$. Let $A \unlhd K \unlhd G$ such
that $K/A$ is a cyclic subgroup of $Z(G/A)$. If
$K/A$ is contained in the kernel of the action
of $G/A$ on $A$, then $K$ is a normal abelian
subgroup of $G$, hence $K = A$ by the
hypotheses on $G$ and $A$. We deduce that
$G/A$ acts freely on $A$. In other words, $A$
is self-centralizing in $G$. Hence, via the
technical lemma \cite[23.5]{Asc86}, we recover
the following well-known result of Roquette.

\begin{thm}
{\rm (Roquette's Classification Theorem.)}
The Roquette $p$-groups are precisely the
following groups.

\smallskip
\noin {\bf (a)} The cyclic group $C_{p^m}$
where $m \geq 0$.

\smallskip
\noin {\bf (b)} The quaternion group $Q_{2^m}$
where $m \geq 3$.

\smallskip
\noin {\bf (c)} The dihedral group $D_{2^m}$
where $m \geq 4$.

\smallskip
\noin {\bf (d)} The semidihedral group
$\SD_{2^m}$ where $m \geq 4$.
\end{thm}

It is worth sketching the content of the invoked
technical lemma because we shall later be needing
some notation concerning automorphisms of cyclic
$2$-groups. (Besides, there is some charm in the
connection between the classical number theory
behind Theorem 2.4 and the algebraic number theory
in Section 5.) Let $v$ be a power of $2$ with
$v \geq 2$. The group $\Aut(C_{4v}) \cong
(\ZZ/4v)^\times \cong C_2 \times C_v$ has precisely
three involutions, namely the elements $b$, $c$,
$d$ which act on a generator $a$ of $C_{4v}$ by
$$b \, : \, a \mapsto a^{-1} \, , \dozspace
  c \, : \, a \mapsto a^{2v + 1} \, , \dozspace
  d \, : \, a \mapsto a^{2v - 1} \, .$$
Any odd square integer is congruent to $1$ modulo
$8$. So $b$ and $d$ cannot have a square root in
$\Aut(C_{4v})$. Therefore $c$ belongs to every
non-trivial subgroup of $\Aut(C_{4v})$ except
for $\la b \ra$ and $\la d \ra$. Now suppose
that $G$ is a $2$-group with a self-centralizing
normal cyclic subgroup $A = \la a \ra$ with index
$|G:A| \geq 4$. The inequality implies that 
$A \cong C_{4v}$ with $v \geq 2$, and moreover,
the image of $G/A$ in $\Aut(A)$ must own the
involution $c : a \mapsto a^{2v + 1}$. Abusing
notation, the subgroup $\la c \ra$ of $\Aut(A)$
lifts to a normal subgroup $\la a, c \ra \cong
\Mod_{8v}$ of $G$. But $\Mod_{8v}$ has a
characteristic subgroup $\la a^{2v}, c \ra
\cong V_4$. We deduce that $G$ is not Roquette.
The rest of the proof of Theorem 2.4 is
straightforward.

Given $H \leq G$, a $\KK H$-irrep $\phi$ and
a $\KK G$-irrep $\psi$, then $\phi$ occurs in
$\res_H^G(\psi)$ if and only if $\psi$ occurs
in $\ind_H^G(\phi)$. When these two equivalent
conditions hold, we say that $\psi$ and $\phi$
{\bf overlap}. The following observation is an
easy consequence of Clifford's Theorem.

\begin{lem}
Suppose that $G$ has a self-centralizing normal
cyclic subgroup $A$.

\smallskip
\noin {\bf (1)} Given a $\KK G$-irrep $\psi$
overlapping with a $\KK A$-irrep $\xi$, then
$\psi$ is faithful if and only if $\xi$ is
faithful.

\smallskip
\noin {\bf (2)} Given a faithful $\KK G$-irrep
$\psi$ overlapping with a faithful $\KK A$-irrep
$\xi$, then $\psi$ is an integer multiple of
$\ind_A^G(\eta)$. Furthermore, $\psi$ is
absolutely irreducible if and only if $\xi$ is
absolutely irreducible, in which case, the
integer multiple is unity.

\smallskip
\noin {\bf (3)} The condition that $\psi$ and
$\xi$ overlap characterizes a bijective
correspondence between the faithful $\KK G$-irreps
$\psi$ and the $G$-conjugacy classes of faithful
$\KK A$-irreps $\xi$.
\end{lem}

Since the Roquette $p$-groups satisfy the
hypothesis of Lemma 2.5, we deduce that any
Roquette $p$-group has a faithful $\KK$-irrep.

We let $n(G)$ denote the exponent of $G$.
Brauer's Splitting Theorem asserts that the
cyclotomic field $\QQ_{n(G)}$ splits for $G$,
that is to say, every $\QQ_{n(G)} G$-irrep is
absolutely irreducible.

\begin{lem} Suppose that $G$ is Roquette. Then

\smallskip
\noin {\bf (1)} The automorphism group
$\Aut(G)$ acts transitively on the faithful
$\KK G$-irreps.

\smallskip
\noin {\bf (2)} The Galois group
$\Aut(\QQ_{n(G)}) = \Gal(\QQ_{n(G)}/\QQ)$ acts
transitively on the faithful $\KK G$-irreps.
The action is such that $({}^\alpha \psi)(g)
= {}^\alpha(\psi(g))$ where $g \in G$ and
$\psi$ is a faithful $\KK G$-irrep.
\end{lem}

\begin{proof}
Write $n = n(G)$. First suppose that $G$ is
cyclic. Then $n = |G|$. Part(1) is clear in
this case. There is a triple of commuting
isomorphisms between the groups $\Aut(G)$
and $(\ZZ/n)^\times$ and $\Aut(\QQ_n)$ such
that, given elements $\aleph$ and $\&$ and
$\alpha$, respectively, then $\aleph
\leftrightarrow \& \leftrightarrow \alpha$
provided $\aleph(g^\&) = g$ and $\alpha(\omega)
= \omega^\&$ where $g \in G$ and $\omega$ is
an $n$-th root of unity. Then
${}^\alpha(\psi(g)) = ({}^\aleph \psi)(g)$.
Thus, the specified action of $\Aut(\QQ_n)$
on the faithful $\KK G$-irreps coincides with
the action via the isomorphism $\Aut(\QQ_n)
\cong \Aut(G)$. Part (2) is now clear in the
case where $G$ is cyclic.

Now suppose that $G$ is non-cyclic. The
classification of the Roquette $p$-groups
implies that $p = 2$ and $G$ is dihedral,
semidihedral or quaternion. So there exists
a cyclic maximal subgroup $A$ and an element
$y \in G - A$ such that either $y^2 = 1$ or
$y^2$ is the unique involution in $A$.
Any automorphism $\aleph$ of $A$ must fix
$y^2$, so $\aleph$ can be extended to an
automorphism $@$ of $G$ such that $@$ fixes
$y$. We have already seen that $\Aut(A)$ acts
transitively on the faithful $\KK A$-irreps.
In view of the bijective correspondence in
Lemma 2.5, $\Aut(A)$ acts transitively on
the faithful $\KK G$-irreps via the
monomorphism $\Aut(A) \ni \aleph \mapsto
@ \in \Aut(G)$. Part (1) follows perforce.
Now suppose that $\aleph$ and $\alpha$ are
corresponding elements of $\Aut(A)$ and
$\Aut(\QQ_n)$. By part (2) of Lemma 2.5
together with the formula for induction of
characters, the faithful $\KK G$-irreps
vanish off $A$. So ${}^\alpha(\psi(g)) =
({}^@ \psi)(g)$ for all $g \in G$. As
before, part (2) follows.
\end{proof}

By the same argument, the conclusions of the
lemma also hold for the modular $p$-groups.

We can now complete the proof of Theorem 2.2
and Proposition 2.3. Above, we showed that
the proposition implies the theorem, and we
also explained how the proposition reduces
to the case where $G$ is Roquette. But that
case of the proposition is weaker than
part (2) of Lemma 2.6. The theorem and the
proposition are now proved.

\begin{cor}
Let $\psi$ be a $\KK G$-irrep. Let $\JJ$ be a
Galois extension of $\QQ[\psi]$. Then
$\Gal(\JJ/\QQ)$ acts transitively on the
on the $\KK G$-irreps that are quasiconjugate
to $\psi$. If $\JJ$ owns primitive $n(G)$-th
roots of unity, then two $\KK G$-irreps
$\psi_1$ and $\psi_2$ lie in the same
$\Gal(\JJ/\QQ)$-conjugacy class if and only
if $\psi_1$ and $\psi_2$ are quasiconjugate.
\end{cor}

\begin{proof}
This follows from Theorem 2.2 by replacing
$\LL/\KK$ with $\KK/\QQ$.
\end{proof}

When $\psi$ and $\psi'$ satisfy the equivalent
conditions in the latest corollary, we say
that $\psi$ and $\psi'$ are {\bf Galois
conjugate}. Thus, we may speak unambiguously
of the Galois conjugates of a given $\KK G$-irrep;
there is no need to specify the Galois extension
and there is no need for the Galois automorphisms
to stabilize $\KK$ nor even to be defined on $\KK$. 
We can now express part (4) of Lemma 2.6 more
succinctly.

\begin{cor}
If $G$ is Roquette, then the faithful $\KK G$-irreps
comprise a single Galois conjugacy class.
\end{cor}

Keeping in mind the above features of Galois
conjugacy, we see that the following invariants
of a $\KK G$-irrep $\psi$ are quasiconjugacy
global invariants. Let $\JJ$ be any field
with characteristic zero. In some of the items
below, it may seem that we have proliferated
notation unnecessarily, and that it would be
simpler to present only the case $\KK = \JJ$.
However, there is a distinction to be made:
the invariants are associated with the field
$\JJ$, whereas the given irrep $\psi$ has
coefficient field $\KK$. In the applications
in Section 5, we shall be mostly concerned with
the cases $\JJ = \QQ$ and $\JJ = \RR$, but the
given irrep $\psi$ will still have coefficients
in arbitrary $\KK$. Recall that $\psi_\JJ$
denotes a $\JJ G$-irrep that is quasiconjugate
to $\psi$. For the first item in the list, we
let $\LL$ be any field extension of $\JJ$.

\smallskip
\noin $\bullet$ {\bf The $\LL/\JJ$-relative Schur
index $m_\JJ^\LL(\psi)$ and the $\LL/\JJ$-relative
order $v_\JJ^\LL(\psi)$.} We define them to be the
positive integers $m$ and $v$, respectively, such
that the $\LL$-linear extension of $\psi$ can be 
written in the form $\LL \psi = m(\psi_1 + ... +
\psi_v)$ where $\psi_1$, $...$, $\psi_v$ are
mutually distinct $\LL G$-irreps. Theorem 2.2
tells us that each $\psi_j$ is a Galois conjugate
of $\psi_\LL$. Schilling's Theorem 5.9 tells us
that $m_\JJ^\LL(\psi) \leq 2$.

\smallskip
\noin $\bullet$ {\bf The endomorphism ring
$\End_{\JJ G}(\psi_\JJ)$.} Strictly speaking,
the invariant here is the isomorphism class
of $\End_{\JJ G}(\psi_\JJ)$ as a $\JJ$-algebra.

\smallskip
\noin $\bullet$ {\bf The class of minimal
splitting fields for $\psi_\JJ$.} Still letting
$\LL$ be an extension field of $\JJ$, the
$\LL$-irrep $\psi_\LL$ is absolutely irreducible
if and only if $\LL$ is a splitting field for
$\End_{\JJ G}(\psi_\JJ)$, or equivalently, $\LL$
is a splitting field for the Wedderburn component
of $\JJ G$ associated with $\psi_\JJ$. When these
equivalent conditions hold, $\LL$ is said to be a
{\bf splitting field} for $\psi$. If furthermore,
the degree $|\II : \JJ|$ is minimal, then $\LL$
is said to be a {\bf minimal splitting field}
for $\psi$.

\smallskip
\noin Let $\MM$ be a splitting field for
$\psi_\JJ$. In the next two quasiconjugacy
global invariants, the stated properties
of $m_\JJ(\psi)$ and $v_\JJ(\psi)$ are
well-known and can be found in Curtis--Reiner
\cite[Section 74]{CR87}.

\smallskip
\noin $\bullet$ {\bf The $\JJ$-relative Schur
index $m_\JJ(\psi)$ and the $\JJ$-relative
order $v_\JJ(\psi)$.} Defined as $m_\JJ(\psi)
= m_\JJ^\MM(\psi)$ and $v_\JJ(\psi) =
v_\JJ^\MM(\psi)$, they are independent of the
choice of $\MM$. We mention that, if $\MM$ is
a minimal splitting field for $\psi_\JJ$, then
its degree over $\JJ$ is $m_\JJ(\psi) v_\JJ(\psi)
= |\MM : \JJ|$.

\smallskip
\noin $\bullet$ {\bf The $\JJ$-relative vertex
field $\VV_\JJ(\psi)$.} This invariant is an
isomorphism class of extension fields of $\JJ$.
It has three equivalent definitions: firstly,
$\VV_\JJ(\psi) = \JJ[\psi_\MM]$; secondly,
$\VV_\JJ(\psi)$ is the centre of the division ring
$\End_{\JJ G}(\psi_\JJ)$; thirdly, $\VV_\JJ(\psi)$
is the centre of the Wedderburn component of
$\JJ G$ associated with $\psi_\JJ$. We mention
that $v_\JJ(\psi) = |\VV_\JJ(\psi) : \JJ|$. In
other words, $m_\JJ(\psi) = |\MM : \VV_\JJ(\psi)|$
when the splitting field $\MM$ is minimal. Also,
$m_\JJ(\psi)$ is the square root of the dimension
of $\End_{\JJ G}(\psi_\JJ)$ over $\VV_\JJ(\psi)$.

\smallskip
Our reason for ploughing through this systematic
notation is that, in Section 5, we shall show
that the above invariants are not merely
quasiconjugacy global invariants. They are also
tight induction invariants. That is to say,
they are genetic invariants. This is a
compelling vindication of the proposed notion
of tight induction. Also, as a speculative
motive for considering the invariants in such
generality, let us suggest the possibility of
a technique whereby assertions pertaining to
arbitrary $\KK$ may be demonstrated by first
dealing with one of the extremal cases
$\KK = \QQ$ or $\KK = \CC$, then establishing
a passage for field extensions with prime
degree, and then arguing by induction on the
length of an abelian Galois group.

However, to characterize the genotype of a given
irrep, we shall only be making use of the cases
$\JJ = \QQ$ and $\JJ = \RR$. Let us list the
genetic invariants that will be of applicable
significance in Section 5. Some of them are
special cases of the above.

\smallskip
\noin $\bullet$ The endomorphism ring
$\End_{\QQ G}(\psi_\QQ)$, which is a
ring well-defined up to isomorphism.

\smallskip
\noin $\bullet$ The class of minimal
splitting fields for $\psi_\QQ$.

\smallskip
\noin $\bullet$ The vertex field $\VV(\psi) =
\QQ[\psi_\CC]$. Besides the three equivalent
definitions above, another characterization
of $\VV(\psi)$ will be given in Proposition
5.10 (and this fourth equivalent definition
supplies a rationale for the terminology.)

\smallskip
\noin $\bullet$ The {\bf exponent} $n(\psi)$,
which we define to be the minimal positive
integer such that $\QQ_{n(\psi)}$ is a
splitting field for $\psi_\QQ$.

\smallskip
\noin $\bullet$ The {\bf Fein field} of $\psi$,
which we define to be the unique subfield
$\Fein(\psi) \leq \QQ_{n(\psi)}$ such that
$\Fein(\psi)$ is a minimal splitting field for
$\psi_\QQ$. The existence and uniqueness of
$\Fein(\psi)$ will be proved in Theorem 5.7. (The
existence can fail for arbitrary finite groups.)

\smallskip
\noin $\bullet$ The Schur index $m(\psi) =
m_\QQ(\psi)$ and the order $v(\psi) =
v_\QQ(\psi)$. We mention that $m(\psi)$ is the
multiplicity of $\psi_\CC$ and $v(\psi)$ is the
number of Galois conjugates of $\psi_\CC$. Also,
$$2 \, \geq \, m(\psi) = |\Fein(\psi) : \VV(\psi)| =
\sqrt{\dim_{\VV(\psi)}(\End_{\QQ G}(\psi_\QQ))} \, .$$

\smallskip
\noin $\bullet$ The vertex set $\Vtx(\psi)$,
which we define to be the transitive
$\Aut(\VV(\psi))$-set consisting of the
$\CC G$-irreps that are quasiconjugate to
$\psi$. We sometimes call these $\CC G$-irreps
the {\bf vertices} of $\psi$. Actually, the
invariant here is the isomorphism class of
$\Vtx(\psi)$ as an $\Aut(\VV(\psi))$-set.
Putting $v = v(\psi)$ and letting $\psi_1$,
$...$, $\psi_v$ be the vertices of $\psi$,
then the $\VV(\psi) \, G$-irreps contained
in $\psi$ can be enumerated as $\psi'_1$,
$...$, $\psi'_v$ in such a way that the
$\VV(\psi)$-linear extension of $\psi$
decomposes as $\VV(\psi) \, \psi = \psi'_1 +
... + \psi'_v$ and the $\CC$-linear extension
of each $\psi'_j$ decomposes as $\CC \psi'_j =
m(\psi) \psi_j$. Thus, $\Aut(\VV(\psi))$
permutes the $\VV(\psi)$-irreps $\psi'_j$ just
as it permutes the vertices $\psi_j$. Note that
$$v(\psi) = |\VV(\psi) : \QQ| = |\Vtx(\psi)| .$$
(Another rationale for the terminology now
becomes apparent.) We point out that, given
any field extension $\II$ of $\VV(\psi)$, then
any automorphism of $\II$ restricts to an
automorphism of $\VV(\psi)$, hence $\Vtx(\psi)$
becomes an $\Aut(\II)$-set.

\smallskip
\noin $\bullet$ The endomorphism algebra
$\Delta(\psi) = \End_{\RR G}(\psi_\RR)$ is called
the {\bf Frobenius--Schur type} of $\psi$.
Understanding $\Delta(\psi)$ to be well-defined
only up to ring isomorphism, then there are only
three possible values, namely $\RR$ and $\CC$ and
$\HH$. The respective values of the pair
$(m_\RR(\psi), v_\RR(\psi))$ are $(1,1)$ and
$(1,2)$ and $(2,1)$. If $\psi$ is given as a
$\KK G$-character $G \rightarrow \KK$, then a
practical way to determine $\Delta(\psi)$ is to
make use of the {\bf Frobenius--Schur} indicator,
which is defined to be the integer
$$\fs(\psi) =
  \frac{1}{|G|} \sum_{g \in G} \psi(g^2).$$
Recall that $\Delta(\psi)$ is $\RR$ or
$\CC$ or $\HH$ depending on whether $\fs(\psi_\CC)
= 1$ or $\fs(\psi_\CC) = 0$ or $\fs(\psi_\CC) =
-1$, respectively. Also, $\fs(\psi) = m_\KK(\psi)
v_\KK(\psi) \fs(\psi_\CC)$. Therefore $\Delta(\psi)$
is $\RR$ or $\CC$ or $\HH$ depending on whether
$\fs(\psi) > 0$ or $\fs(\psi) = 0$ or $\fs(\psi)
< 0$, respectively. The genetic invariance of
$\Delta(\psi)$ is implicit in Yamanda--Iida
\cite[5.2]{YI93}.

\section{Tight induction}

Let us repeat the most important definition in
this paper. Consider a subgroup $H \leq G$, a
$\KK G$-irrep $\psi$ and a $\KK H$-irrep $\phi$
such that $\psi$ is induced from $\phi$. When
no Galois conjugate of $\phi$ occurs in
$\res_H^G(\psi) - \phi$, we say that $\psi$ is
{\bf tightly induced} from $\phi$ and, abusing
notation, we also say that the induction $\psi
= \ind_H^G(\phi)$ is {\bf tight}. As we noted
in Section 1, the definition extends in the
evident way to induction from subquotients. The
tightness condition can usefully be divided into
two parts, as indicated in the next two lemmas.

\begin{lem}
{\rm (Shallow Lemma)} Given $H \leq G$, and
$\KK G$-irrep $\psi$ induced from a $\KK H$-irrep
$\phi$, then the following conditions are equivalent:

\smallskip
\noin {\bf (a)} The multiplicity of $\phi$ in
$\res_H^G(\psi)$ is $1$.

\smallskip
\noin {\bf (b)} The division rings
$\End_{\KK H}(\phi)$ and $\End_{\KK G}(\psi)$
have the same $\KK$-dimension.

\smallskip
\noin {\bf (c)} As $\KK$-algebras,
$\End_{\KK H}(\phi)$ and $\End_{\KK G}(\psi)$
are isomorphic.
\end{lem}

\begin{proof}
As $\KK$-vector spaces, we embed $\phi$ in $\psi$
via the identifications $\phi = 1 \otimes \phi$ and
$\psi = \bigoplus_{gH \subseteq G} g \otimes \phi$.
We embed the $\KK$-algebra $\cD = \End_{\KK H}(\phi)$
in the $\KK$-algebra $\cE = \End_{\KK G}(\psi)$ by
letting $\cD$ kill the module $\theta = \sum_{gH
\subseteq G - H} g \otimes \phi$. The relative trace
map $\tr_H^G : \End_{\KK H}(\psi) \rightarrow \cE$
restricts to an $\KK$-algebra monomorphism $\nu :
\cD \rightarrow \cE$. So conditions (b) and (c)
are both equivalent to the condition that $\nu$ is
a $\KK$-algebra isomorphism.

Suppose that (a) holds. Then any $\KK H$-endomorphism
of $\psi$ restricts to a $\KK H$-endomorphism of
$\phi$. In particular, any element $\epsilon \in \cE$
restricts to an element $\mu(\epsilon) \in \cD$. We
have defined a $\KK$-algebra map $\mu : \cE
\rightarrow \cD$. From the constructions, we see that
$\mu \nu$ is the identity map on $\cD$. So $\mu$ is
surjective. But $\cD$ is a division ring, so $\mu$ is
injective. Hence $\mu$ and $\nu$ are mutually inverse
$\KK$-algebra isomorphisms. We have deduced (b) and (c).

Now suppose that (a) fails. Let $\phi'$ be a
$\KK H$-submodule of $\theta$ such that $\phi'
\cong \phi$. Let $\beta$ be a $\KK H$-endomorphism
of $\psi$ such that $\beta$ kills $\theta$ and
$\beta$ restricts to a $\KK H$-isomorphism $\phi
\rightarrow \phi'$. Let $\gamma = \tr_H^G(\beta)$.
Then $\beta$ and $\gamma$ have the same action
on $\phi$. In particular, $\gamma$ restricts to
an isomorphism $\phi \rightarrow \phi'$. On the
other hand, any element $\delta \in \cD$ has the
same action on $\phi$ as $\nu(\delta)$.
In particular, $\nu(\delta)$ restricts to a
$\KK H$-automorphism of $\phi$. Therefore $\gamma
\in \cE - \nu(\cD)$ and $\nu$ is not surjective.
We have deduced that (b) and (c) fail.
\end{proof}

\begin{lem}
{\rm (Narrow Lemma)} Given $H \leq G$, and
$\KK G$-irrep $\psi$ induced from a $\KK H$-irrep
$\phi$, then the following conditions are equivalent:

\smallskip
\noin {\bf (a)} No distinct Galois conjugate of
$\phi$ occurs in $\res_H^G(\psi)$.

\smallskip
\noin {\bf (b)} The condition $\psi' =
\ind_H^G(\phi')$ describes a bijective correspondence
between the Galois conjugates $\psi'$ of $\psi$ and
the Galois conjugates $\phi'$ of $\phi$.

\smallskip
\noin {\bf (c)} We have $\QQ[\phi] = \QQ[\psi]$.
\end{lem}

\begin{proof}
The equivalence of (a) and (b) is clear by
Frobenius Reciprocity. By the standard formula
for the values of an induced character, $\QQ[\phi]
\geq \QQ[\psi]$. The fields $\QQ[\phi]$ and
$\QQ[\psi]$ are subcyclotomic, so the field
extension $\QQ[\phi]/\QQ[\psi]$ is Galois.
Conditions (b) and (c) are both equivalent to
the condition that no Galois automorphism moves
$\phi$ and fixes $\psi$. 
\end{proof}

When the equivalent conditions in Lemma 3.1 hold,
we say that $\psi$ is {\bf shallowly} induced
from $\phi$. When the equivalent conditions in
Lemma 3.2 hold, we say that $\psi$ is {\bf
narrowly} induced from $\phi$. The induction
$\psi = \ind_H^G(\phi)$ is tight if and only
if it is shallow and narrow. In the special
case $\KK = \QQ$, the narrowness condition is
vacuous: an induction of rational irreps $\psi
= \ind_H^G(\phi)$ is tight if and only if
$\End_{\QQ G}(\psi) \cong \End_{\QQ H}(\phi)$.
The definition of tight induction in the case
$\KK = \QQ$ is due to Witt \cite{Wit52}. At
the other extreme, when $\KK$ is algebraically
closed, the shallowness condition is vacuous:
an induction of complex irreps $\psi =
\ind_H^G(\phi)$ is tight if and only if
$\QQ[\psi] = \QQ[\phi]$. The definition of
tight induction in the case $\KK = \CC$ is
due to Kronstein and, independently, to
Iida-Yamanda \cite{IY92}.

\begin{rem}
Let $H \leq L \leq G$. Let $\phi$ and $\theta
= \ind_H^L(\phi)$ and $\psi = \ind_L^G(\theta)$
be $\KK$-irreps of $H$ and $L$ and $G$,
respectively. If any two of the inductions
$\theta = \ind_H^L(\phi)$ and $\psi =
\ind_L^G(\theta)$ and $\psi = \ind_H^G(\phi)$
are shallow, then all three are shallow. If
any two of the inductions are narrow, then all
three are narrow. If any two of them are tight,
then all three are tight.
\end{rem}

The remark is obvious. It tells us, in particular,
that tight induction is transitive. In fact, given
a $\KK G$-irrep, then there is a $G$-poset whose
elements are the pairs $(H, \phi)$ such that $H
\leq G$ and $\phi$ is a $\KK H$-irrep from which
$\psi$ is tightly induced. The partial ordering is
such that $(H, \phi) \leq (L, \theta)$ provided
$H \leq L$ and $\theta$ is induced from $\phi$
(whereupon, by the remark, $\theta$ is tightly
induced from $\phi$). The Genotype Theorem 1.1
(proved in the next section) implies that the
minimal elements of the $G$-poset are the pairs
$(H, \phi)$ such that $H/\Ker(\phi)$ is Roquette.

\begin{thm}
Let $H \leq G$ and let $\psi$ be a $\KK G$-irrep
induced from a $\KK H$-irrep $\phi$. Let $\JJ$
be a subfield of $\KK$. Let $\LL$ be a field
extension of $\KK$. Then the following conditions
are equivalent:

\smallskip
\noin {\bf (a)} The $\JJ G$-irrep containing
$\psi$ is tightly induced from the $\JJ H$-irrep
containing $\phi$.

\smallskip
\noin {\bf (b)} $\psi$ is tightly induced
from $\phi$.

\smallskip
\noin {\bf (c)} There is a bijective
correspondence between the $\LL G$-irreps $\psi'$
contained in $\psi$ and the $\LL H$-irreps
$\phi'$ contained in $\phi$. The correspondence
is characterized by the condition that $\psi'$
is tightly induced from $\phi'$.
\end{thm}

\begin{proof}
When extending the
coefficient field for finite-dimensional modules,
the extension of the hom-space is the hom-space
of the extensions. So the $\JJ G$-irrep $\psi''$
containing $\psi$ must overlap with $\JJ H$-irrep
$\phi''$ containing $\phi$. By Theorem 2.2,
$$\KK \psi'' = m_\JJ^\KK(\psi) \sum_{\alpha \in
  \Gal(\QQ[\psi]/\QQ[\psi''])} {}^\alpha \psi
  \, , \dozspace \dozspace \KK \phi'' =
  m_\JJ^\KK(\phi) \sum_{\beta \in
  \Gal(\QQ[\phi]/\QQ[\phi''])} {}^\beta \phi \, .$$
Suppose that (b) holds. Then $\QQ[\psi] =
\QQ[\phi]$. Since $\phi''$ occurs in
$\res_H^G(\psi'')$, since ${}^\beta \psi$ is the
unique Galois conjugate of $\psi$ overlapping with
${}^\beta \phi$, and since ${}^\beta \phi$ occurs
only once in the restriction of ${}^\beta \psi$,
the set of indices $\beta$ must be contained in the
set of indices $\alpha$, and $m_\JJ^\KK(\phi) \leq
m_\JJ^\KK(\psi)$. Since $\psi''$ occurs in
$\ind_H^G(\phi'')$, since ${}^\alpha \phi$ is the
unique Galois conjugate of $\phi$ overlapping with
${}^\alpha \psi$, and since ${}^\alpha \phi$
induces to ${}^\alpha \psi$, the set of indices
$\alpha$ must be contained in the set of indices
$\beta$, and $m_\JJ^\KK(\phi) \geq m_\JJ^\KK(\psi)$.
So the two sets of indices coincide. That is to say,
$\QQ[\psi''] = \QQ[\phi'']$. Furthermore,
$m_\JJ^\KK(\phi) = m_\JJ^\KK(\psi)$. It follows
that $\phi''$ induces to $\psi''$. Also, $\phi''$
occurs only once in the restriction of $\psi$, in
other words, the induction from $\phi''$ to
$\psi''$ is shallow. We have already observed that
$\QQ[\psi''] = \QQ[\phi'']$, in other words, the
induction is narrow. Thus, (b) implies (a).

Still assuming (b), we now want (c). Each
$\LL H$-irrep contained in $\phi$ must
overlap with at least one $\LL G$-irrep contained
in $\psi$, and each $\LL G$-irrep contained
in $\psi$ must overlap with at least one
$\LL H$-irrep contained in $\phi$. Applying
the functor $\LL \otimes_\KK \dash$ to the
$\KK$-algebra isomorphism $\End_{\KK H}(\phi) \cong
\End_{\KK G}(\psi)$, we obtain an isomorphism
of semisimple rings $\End_{\LL H}(\LL \phi) \cong
\End_{\LL G}(\LL \psi)$. The number of Wedderburn
components of this semisimple ring is equal to the
number of distinct $\LL H$-irreps contained in
$\phi$, and it is also equal to the number of
distinct $\LL G$-irreps contained in $\psi$.
By Theorem 2.2, all the $\LL H$-irreps contained
in $\phi$ have the same multiplicity $m$, and all
the $\LL G$-irreps contained in $\psi$ have the
same multiplicity $n$. So the Wedderburn components
all have the same degree as matrix algebras over
their associated division rings, and $m = n$. It
follows that there is a bijection $\psi'
\leftrightarrow \phi'$ whereby $\psi'$ is shallowly
induced from $\phi'$. We must show that the
induction is narrow. By the formula for induction
of characters, the field $\QQ[\phi']$ (which is
independent of the choice of $\phi'$) contains the
field $\QQ[\psi']$ (which is independent of the
choice of $\psi'$). By Theorem 2.2, the number of
distinct $\LL H$-irreps contained in $\phi$ is
equal to the order of the Galois group
$\Gal(\QQ[\phi']/\QQ[\phi])$ while the number of
distinct $\LL G$-irreps contained in $\psi$ is
the order of $\Gal(\QQ[\psi']/ \QQ[\psi])$. But
we already know that these two numbers are equal.
Moreover, $\QQ[\psi] = \QQ[\phi]$ as part
of the hypothesis on $\psi$ and $\phi$. Therefore
$\QQ[\psi'] = \QQ[\phi']$. We have gotten (c)
from (b). To obtain (b) from (a) or from (c), we
interchange the extensions $\LL/\KK$ and $\KK/\JJ$.
\end{proof}

For facility of use, it is worth restating the
theorem.

\begin{thm}
{\rm (Field-Changing Theorem)}
Let $H \leq G$ and let $\psi$ be a $\KK G$-irrep
induced from a $\KK H$-irrep $\phi$. Let $\LL$ be
any field having characteristic zero. Then the
following conditions are equivalent:

\smallskip
\noin {\bf (a)} $\psi$ is tightly induced
from $\phi$.

\smallskip
\noin {\bf (b)} $\psi_\QQ$ is tightly induced
from $\phi_\QQ$.

\smallskip
\noin {\bf (c)} The $\LL G$-irreps $\psi'$ that
are quasiconjugate to $\psi$ are in a bijective
correspondence with the $\LL H$-irreps $\phi'$
that are quasiconjugate to $\phi$. They
correspond $\psi' \leftrightarrow \phi'$
when $\psi'$ is tightly induced from $\phi'$.
\end{thm}

The theorem tells us that, in some sense,
the genetic theory is independent of the
coefficient field $\KK$. Condition (b) is
a useful theoretical criterion for tightness
of a given induction $\psi = \ind_H^G(\phi)$.
It sometimes allows us to generalize
immediately from the case $\KK = \QQ$ to the
case where $\KK$ is arbitrary; see the next
section. However, the rational irreps of a
given finite $p$-group are usually very
difficult to determine. For explicit analysis
of concrete examples, a more practical criterion
for tightness is given by the following corollary.

\begin{cor}
Let $H \leq G$ and let $\psi$ be a $\KK G$-irrep
induced from a $\KK H$-irrep $\phi$. Then the
vertex fields satisfy the inequality $\VV(\psi)
\leq \VV(\phi)$, and equality holds if and only
if the induction is tight.
\end{cor}

\begin{proof}
By passing from $\KK$ to the algebraic closure
of $\KK$, thence to the splitting field
$\QQ_{n(G)}$, thence to $\CC$, we see that
$\psi' = \ind_H^G(\phi')$ for some
complex irreps $\psi'$ and $\phi'$ quasiconjugate
to $\psi$ and $\phi$. By the formula for induction
of characters, the vertex field $\VV(\psi) =
\QQ[\psi']$ is contained in the vertex field
$\VV(\phi) = \QQ[\phi']$. By the Field-Changing
Theorem, $\psi$ is tightly induced from $\phi$ if
and only if $\psi'$ is tightly induced from
$\phi'$. For complex irreps, tight induction is
just narrow induction.
\end{proof}

Let us give an example. For $n \geq 5$, we define
$\DD_{2^n} = V_4 \ltimes C_{2^{n-2}}$ as a
semidirect product where $V_4$ acts faithfully.
The $2$-group $\DD_{2^n}$ has generators $a$,
$b$, $c$, $d$ with relations
$$a^{4u} = b^2 = c^2 = d^2 = bcd = 1 \, ,
  \;\;\;\;\;\; bab^{-1} = a^{-1} \, , \;\;\;\;\;\;
  cac^{-1} = a^{2u+1} \, , \;\;\;\;\;\;
  dad^{-1} = a^{2u-1}$$
where $u = 2^{n-4}$. Fixing $n$, let us write
$\DD = \DD_{2^n}$. Let $\omega$ be a primitive
$4u$-th root of unity. The subgroup $A =
\la a \ra \cong C_{4u}$ has complex irrep $\eta$
such that $\eta(a) = \omega$. The subgroup $D'
= \la b, a \ra \cong D_{8u}$ has a real abirrep
(absolutely irreducible representation) $\phi'$
such that $\CC \phi' = \ind_A^{D'}(\eta)$. Using
Lemma 2.5, we see that $\DD$ has a real abirrep
$\chi = \ind_{D'}^{DD}(\phi')$, and furthermore,
the faithful real irreps of $\DD$ are precisely
the Galois conjugates of $\chi$.
The induction from $\phi'$ to $\chi$ is not tight.
One way to see this is to calculate the vertex
fields of $\phi'$ and $\chi$ over $\QQ$. The
character values vanish of $A$ and, given an
integer $k$, we have $\phi'(a^k) = \omega^k +
\omega^{-k}$ and $\chi(a^k) = \omega^k +
\omega^{k(2u-1)} + \omega^{k(2u+1)} + \omega^{-k}$.
But $\omega^{2u} = -1$ so $\chi$ vanishes off
$\la a^2 \ra$ and $\chi(a^{2k}) = 2(\omega^{2k}
+ \omega^{-2k})$. Since $\phi'$ and $\chi$ are
absolutely irreducible, the vertex fields are
$\VV(\phi') = \QQ[\phi'] = \QQ[\omega +
\omega^{-1}]$ and $\VV(\chi) = \QQ[\chi] =
\QQ[\omega^2 + \omega^{-2}]$; the former is a
quadratic extension of the latter. Alternatively,
to see directly that the induction is shallow but
not narrow, observe that $\res_{D'}^{DD}(\chi) =
\phi' + {}^\alpha \phi'$ where $\alpha$ is any
Galois automorphism sending $\omega$ to
$-\omega$ or to $-\omega^{-1}$.
However, $\chi$ is tightly induced from a
strict subgroup. Consider the subgroups
$$C = \la c \ra \cong C_2 \, , \dozspace
  D = \la a^2, b \ra \cong D_{4u}  \, , \dozspace
  T = \la a^2, b, c \ra = C \times D \, .$$
Let $\phi$ be the real abirrep of $T$ such
that $\phi(a^2) = \omega^2 + \omega^{-2}$
and $\phi(b) = 0$ and $\phi(c) = 2$. Thus,
$\Ker(\phi) = C$ and $\phi$ is the inflation
of a faithful real abirrep of the group
$T/C \cong D_{4u}$. Direct calculation yields
$\chi = \ind_T^{DD}(\phi)$. This induction is
tight because, by the absolute irreducibility
of $\phi$, the vertex field is $\VV(\phi) =
\QQ[\phi] = \QQ[\omega^2 + \omega^{-2}] =
\VV(\chi)$. We shall be returning to this
example at the end of the next section.

\section{Genotypes and germs}

Let us begin by quickly proving the Genotype
Theorem 1.1. The Field-Changing Theorem 3.5
implies that, for any subgroup $H \leq G$, a
given $\KK G$-irrep $\psi$ is tightly induced
from $H$ if and only if the $\QQ G$-irrep
$\psi_\QQ$ is tightly induced from $H$. Moreover,
for any $\KK H$-irrep $\phi$ that tightly induces
to $\psi$, the $\QQ H$-irrep $\phi_\QQ$ tightly
induces to $\phi_\QQ$. Letting $K$ be the kernel
of $\phi$, then $K$ is the kernel of any Galois
conjugate of $\phi$ and, via Theorem 2.2, $K$ is
the kernel of $\phi_\QQ$. We deduce that the
subquotients from which $\psi$ is tightly induced
coincide with the subquotients from which
$\psi_\QQ$ is tightly induced. The Genotype Theorem
thus reduces to the case $\KK = \QQ$. In that
special case, the theorem was obtained by Bouc
\cite[3.4, 3.6, 3.9, 5.9]{Bou04}. Alternatively,
a similar use of the Field Changing Theorem
reduces to the case $\KK = \CC$, and in that
special case, the theorem was obtained by
Kronstein \cite[2.5]{Kro66}. The proof of
the Genotype Theorem is complete.

The existence half of the Genotype Theorem
is equivalent to the following result, which
is due to Roquette \cite{Roq58} in the case
$\KK = \QQ$ and to Kronstein \cite{Kro66}
in the case $\KK = \CC$.

\begin{thm}
Given a $\KK G$-irrep $\psi$, then is not
tightly induced from any strict subquotient
of $G$ if and only if $G$ is Roquette and
$\psi$ is faithful.
\end{thm}

We shall give a direct proof of Theorem 4.1
without invoking the Field-Changing Theorem.
The direct proof will yield a recursive
algorithm for finding the genotype and a
germ for a given irrep. The argument
is adapted from Hambleton--Taylor--Williams
\cite{HTW90} and Bouc \cite{Bou04}.
Before presenting two preparatory lemmas, let us
make a preliminary claim: supposing that $G$ is
non-cyclic and abelian, then $G$ has no faithful
$\KK$-irreps. To demonstrate the claim, consider
a $\KK G$-irrep $\psi$. Letting $\LL$ be a
splitting field for $\KK$, then $\LL \psi$ is
a direct sum of mutually Galois conjugate
$\LL G$-irreps. All of those $\LL G$-irreps
have the same kernel $K$. The hypothesis that
$G$ is abelian implies that the $\LL G$-irreps
in question are $1$-dimensional, hence $G/K$ is
cyclic. The hypothesis that $G$ is non-cyclic
implies that $K \neq 1$. But $K$ must also be
the kernel of $\psi$. Therefore $\psi$ is
non-faithful.

\begin{lem}
Suppose that $G$ is non-Roquette and that there
exists a faithful $\KK G$-irrep $\psi$. Then there
exists a normal subgroup $E$ of $G$ such that
$E \cong C_p \times C_p$ and $E \cap Z(G) \cong
C_p$. For any such $E$, the subgroup $T = C_G(E)$
is maximal in $G$. Letting $\phi$ be any
$\KK T$-irrep overlapping with $\psi$, then
$\psi$ is tightly induced from $\phi$.
\end{lem}

\begin{proof}
The argument is essentially in \cite[2.16]{HTW90}
and \cite[3.4]{Bou04}, but we must reproduce the
constructions in order to check the tightness
of the induction. First observe that, given any
normal non-cyclic abelian subgroup $A$ of $G$,
then the restriction of $\psi$ to $A$ is
faithful, whence the preliminary claim tells us
that any $\KK A$-irrep overlapping with $\psi$
must be non-inertial. The center $Z(G)$ is cyclic
because every $\KK Z(G)$-irrep is inertial in
$G$. Let $Z$ be the subgroup of $Z(G)$ with order
$p$. let $B$ be the maximal elementary abelian
subgroup of $A$. Then $Z \lhd B \unlhd G$ and
$B/Z$ intersects non-trivially with the centre
of $G/Z$, so there exists an intermediate
subgroup $Z \leq E \leq B$ such that $E/Z$ is
a central subgroup of $G/Z$ with order $p$.
Plainly, $E$ satisfies the required conditions.
The non-trivial $p$-group $G/T$ embeds in the
group $\Aut(C_p \times C_p) = \GL_2(p)$, which
has order $p(p-1)(p^2 - 1)$. So $T$ is maximal
in $G$.

Let $\epsilon_1$ be a $\KK E$-irrep overlapping
with $\psi$. The preliminary claim implies that
the inertia group of $\epsilon_1$ is a strict
subgroup of $G$. On the other hand, the inertia
group must contain the centralizer $T$ of $E$.
But $T$ is maximal. So $T$ is the inertia group
of $\epsilon_1$. The $\KK E$-irreps overlapping
with $\psi$ are precisely the $G$-conjugates of
$\epsilon$, and we can number them as $\epsilon_1$,
$...$, $\epsilon_p$ because $p = |G : T|$. The
proof of the preliminary claim reveals that
the kernels of $\epsilon_1$, $...$, $\epsilon_p$
are mutually distinct; the kernels are non-trivial
yet their intersection is trivial. In particular,
$\epsilon_1$, $...$, $\epsilon_p$ belong to
mutually distinct Galois conjugacy classes.
By Clifford theory, $\res_T^G(\psi) = \phi_1 +
... + \phi_p$ as a direct sum of $\KK T$-irreps
such that each $\phi_j$ restricts to a
multiple of $\epsilon_j$. Therefore, $\phi_1$,
$...$, $\phi_p$ are mutually distinct and, in
fact, they belong to mutually distinct Galois
conjugacy classes. It follows that each $\phi_j$
induces tightly to $\psi$.
\end{proof}

\begin{lem}
Let $A$ be a self-centralizing normal cyclic
subgroup of $G$ and let $A \leq H < G$. Then no
faithful $\KK G$-irrep is tightly induced from $H$.
\end{lem}

\begin{proof}
Deny, and consider a faithful $\KK G$-irrep $\psi$
that is tightly induced from a $\KK H$-irrep
$\phi$ of $H$. By Remark 3.3, we may assume that
$H$ is maximal in $G$. In particular, $H \unlhd G$.
So $\res_H^G(\psi) = \phi_1 + ... + \phi_p$ as a
sum of $G$-conjugates of $\phi$. Since $A$ is
self-centralizing in both $G$ and $H$, Lemma 2.5
implies that $\phi_1$, $...$, $\phi_p$ are
faithful. Lemma 2.6 implies that $\phi_1$, $...$,
$\phi_p$ are Galois conjugates. This contradicts
the tightness of the induction from $\phi$.
\end{proof}

In one direction, Theorem 4.1 is immediate from
Lemma 4.2. To complete the direct proof of the
theorem, it remains only to show that, supposing
$G$ is Roquette and letting $\psi$ be a faithful
$\KK G$-irrep, then $\psi$ is not tightly induced
from a strict subgroup. Our argument is close to
\cite[2.15]{HTW90}, but with some modification
(their appeal to the uniqueness of the ``basic
representation'' does not generalize). For a
contradiction, suppose that $\psi$ is tightly
induced from a $\KK H$-irrep $\phi$ where
$H < G$. Again, by Remark 3.3, we may assume
that $|G : H| = p$. By Roquette's Classification
Theorem 2.4, $G$ has a self-centralizing cyclic
subgroup $A$ with index $1$ or $p$. Plainly,
$G$ cannot be cyclic. So $|G : A| = p$. By
Lemma 4.3, $H \neq A$. So the subgroup $B =
A \cap H$ has index $p^2$ in $G$.

First suppose that $B$ is not self-centralizing
in $H$. Then $H$ must be abelian. But $G$ is
Roquette, hence $H$ is cyclic. Also, $G$ is
non-abelian, so $H$ is self-centralizing.
This contradicts Lemma 4.3. Now suppose that
$B$ is self-centralizing in $H$. By Lemma 2.5,
there exists a faithful
$\KK B$-irrep $\xi$ such that $\ind_B^H(\xi)$
is a multiple of $\phi$. Letting $\zeta =
\ind_B^A(\xi)$, then $\ind_A^G(\zeta)$ is a
multiple of $\psi$. Every $\KK A$-irrep
occurring in $\zeta$ must also occur in
$\res_A^G(\psi)$. But $\zeta$ is induced
from $B$, so some non-faithful $\KK A$-irrep
$\eta$ must occur in $\zeta$. Perforce,
$\eta$ occurs in $\res_A^G(\psi)$. This
contradicts part (1) of Lemma 2.5. The
direct proof of Theorem 4.1 is finished.

Lemma 4.2 gives an algorithm for finding a
genetic subquotient and a germ. First we
replace $G$ with $G/\Ker(\psi)$ to reduce to
the case where $\psi$ is faithful. If
$G/\Ker(\psi)$ is Roquette, then $G/\Ker(\psi)$
is a genetic subquotient, $\psi$ is a germ,
and the algorithm terminates. Otherwise, in
the notation of the lemma, we replace $G$ and
$\psi$ with $T$ and $\phi$, respectively,
and we repeat the process.

By the way, the noncyclic abelian subgroup
$E$ is central in $T$, so the $\KK T$-irrep
$\phi$ is never faithful, and we deduce the
second part of the following incidental
corollary. The first part of the corollary
is immediate from the Genotype Theorem 1.1.

\begin{cor}
Let $H/K$ and $H'/K'$ be genetic factors for
the same $\KK G$-irrep. Then $|H| = |H'|$ and
$|K| = |K'|$. Furthermore, $|G : H| \leq |K|$.
\end{cor}

Let us end this section with a reassessment
of the example $\DD = \DD_{2^n} = \DD_{16u}$,
which was discussed at the end of the previous
section. We employ the same notation as before.
Recall that, although the $\RR \, \DD$-irrep
$\chi$ is induced from the subgroup $D' \cong
D_{4u}$, the induction is not tight. The
failure of tightness can now be seen straight
from Lemma 4.3 because $D'$ contains the
self-centralizing normal cyclic subgroup $A$.
We have already seen that $\chi = \ind_T^G(\phi)$
and that $\phi$ is inflated from a faithful
$\RR$-irrep of the subquotient $T/C \cong
D_{4u}$. So, if $n \geq 6$, then $T/C$ is a
genetic subquotient and $\phi$ is a germ. In
particular, the genetic type is $\Typ(\chi) =
D_{2^{n-2}}$, except in the case $n = 5$, and
in that case, $\Typ(\chi) = C_2$. But let us
recover these conclusions from the algorithm in
a methodical way. As we noted in Section 2, the
$2$-group $\Mod_{8u} = \la a, c \ra$ has a
characteristic subgroup $E = \la a^{2u}, c
\ra \cong V_4$. Treating $\Mod_{8u}$ as
a maximal subgroup of $\DD$, then $E$ is
normal in $\DD$, and the subgroup $T = C_G(E)$
and the irrep $\phi$ that appear in
Lemma 4.2 coincide with the subgroup $T =
\la a^2, b, c \ra$ and the irrep $\phi$ which
we considered at the end of Section 3. Noting
that $C = \Ker(\phi)$, we again arrive at the
conclusion that, if $n \geq 6$ then $T/C$
is a genetic subquotient and $\phi$ is a germ.
Of course, when $n = 5$, the algorithm
continues, the second iteration replacing the
faithful $\RR D_8$-irrep with the faithful
$\RR C_2$-irrep. Let us point out that, in
Section 3, we calculated the vertex fields
$\chi$ and $\phi$ in order to show that the
induction $\chi = \ind_{T/C}^{DD}(\phi)$ is
tight. We have now dispensed with that trip,
and the tightness has been delivered to us
as part of the conclusion of Lemma 4.2.

\section{Characterizations of the genotype}

In the first movement, we shall confine our
attention to the Roquette $p$-groups. For
those $p$-groups, we shall calculate some of
the invariants that were listed in Section 2.
In the second movement, we shall show that all
the invariants listed in Section 2 are genetic
invariants. We shall also terminate a couple
of loose-ends concerning well-definedness.
The third movement will address two questions
that were raised in Section 1: How can the
genotype $\Typ(\psi)$ be ascertained from easily
calculated genetic invariants such as the order
$v(\psi)$, the vertex set $\Vtx(\psi)$, the
vertex field $\VV(\psi)$, the Frobenius--Schur
type $\Delta(\psi)$? How can $\Typ(\psi)$ be
used to ascertain less tractable genetic
invariants such as the exponent $n(\psi)$,
the minimal splitting fields, the Fein field?

To open the first movement, let us observe that,
when $p$ is odd, there is nothing much to say,
as in the next lemma. Note that, given a
$\KK G$-irrep $\psi$, then $\VV(\psi)$ is a
splitting field for $\psi_\QQ$ if and only if
$m(\psi) = 1$. For the time-being, we shall
understand a Fein field for $\psi$ to be a
field that is both a subfield of $\QQ_{n(\psi)}$
and also a splitting field for $\psi_\QQ$. When
we have established the existence and uniqueness
of the Fein field in Theorem 5.7, we shall be at
liberty to write the Fein field as $\Fein(\psi)$.

\begin{lem}
Suppose that $p$ is odd. Let $m$ be a positive
integer. Let $\psi$ be a faithful $\KK$-irrep of
the cyclic group $C_{p^m}$. Then the order of
$\psi$ is $v(\psi) = p^m - p^{m-1}$. The
exponent of $\psi$ is $n(\psi) = p^m$. The
unique minimal splitting field for $\psi_\QQ$
is the unique Fein field for $\psi$, and it
coincides with the vertex field $\VV(\psi) =
\QQ_{p^m}$. The Schur index is $m(\psi) = 1$.
The Frobenius--Schur type is $\Delta(\psi) =
\CC$. The vertex set $\Vtx(\psi)$ is free and
transitive as permutation set for the Galois
group $\Aut(\QQ_{p^m}) = \Gal(\QQ_{p^m})
\cong \Aut(C_{p^m}) \cong (\ZZ/p^m)^\times$.
\end{lem}

We refrain from a systematic discussion of
the groups $C_1$ and $C_2$. The only Roquette
$p$-groups left are the Roquette $2$-groups
with more than one faithful complex irrep.
These will be covered by the next four lemmas.
The lemmas are inevitable exercises, hence
they are well-known. If one could collate a
trawl of citations encompassing all the
conclusions, then that would be a gnomic
achievement. We mention that some of the
material --- including a rather different
discussion of minimal splitting fields for
the quaternion groups --- can be found in
Leedham-Green--McKay \cite[10.1.17]{LM02}.

Let us throw some more notation.
We shall be making use of the matrices
$$B = \openmat 1 & 0 \\ 0 & \;\;\; -1 \closemat
  \; , \dozspace D = \openmat 0 & \;\;\; 1 \\
  1 & 0  \closemat \; , \dozspace X = \openmat
  0 & \;\;\; -1 \\ 1 & 0 \closemat \; .$$
We define $\ex(t) = e^{2 \pi i t}$ and
$\cs(t) = \cos(2 \pi t)$ and $\sn(t) =
\sin(2 \pi t)$ for $t \in \RR$. The matrices
$$R(t) = \openmat \cs(t) & \;\; -\sn(t) \\
  \sn(t) & \cs(t) \closemat \; , \;\;\;\;\;\;
  I(t) = i \openmat \sn(t) & \;\; \cs(t) \\
  -\cs(t) & \sn(t) \closemat \; , \;\;\;\;\;\;
  S(t) = \openmat \cs(t) \; & \; i \, \sn(t) \\
  i \, \sn(t) & \cs(t) \; \closemat$$
satisfy the relation $R(t + t') = R(t) R(t')$
and similarly for $I(t+t')$ and $S(t+t')$. Let
$$A_s(t) = \openmat \cs(t) + i \, \sn(t)/\cs(s)
  & \sn(t) \sn(s)/\cs(s) \;\;\; \\ \sn(t) \sn(s)/\cs(s)
  \;\;\; & \;\;\; \cs(t) - i \, \sn(t)/\cs(s) \closemat$$
where $s \in \RR$. By direct calculation,
$A_s(t + t') = A_s(t) A_s(t')$.

Let $v$ be a power of $2$ with $v \geq 2$. For
convenience, we embed $\QQ_{4v}$ in $\CC$ by
making the identification $\QQ_{4v} = \QQ[\omega]$
where $\omega = \ex(1/4v)$. The Galois group
$$\Aut(\QQ_{4v}) = \Gal(\QQ_{4v} : \QQ) \cong
  \Aut(C_{4v}) \cong (\ZZ/4v)^\times \cong
  C_2 \times C_v$$
has precisely $3$ involutions, namely $\beta$,
$\gamma$, $\delta$ which act on $\QQ_{4v}$ by
$$\beta(\omega) = \omega^{-1} \; , \dozspace
  \gamma(\omega) = \omega^{2v+1} = -\omega
  \; , \dozspace \delta(\omega) = \omega^{2v-1}
  = - \omega^{-1} \; .$$
For a subgroup $\cH \leq \Aut(\QQ_{4v})$, we
let $\Fix(\cH)$ be the intermediate subfield
$\QQ \leq \Fix(\cH) \leq \QQ_{4v}$ fixed by
$\cH$. A straightforward application of the
Fundamental Theorem of Galois Theory shows
that $\QQ_{4v}$ has precisely $3$ maximal
subfields, namely

\smallskip
${\displaystyle \dozspace \;\;\;\;\;\;
\QQ_{4v}^\RR = \Fix \la \beta \ra =
  \QQ[\omega + \omega^{-1}] = \QQ[\cs(r/4v)]
  = \QQ[\sn(r/4v)] = \RR \cap \QQ_{4v} \, ,}$

\smallskip
${\displaystyle \dozspace \;\;\;\;\;\;
\QQ_{2v} = \Fix \la \gamma \ra =
  \QQ[\omega^2] \, ,}$

\smallskip
${\displaystyle \dozspace \;\;\;\;\;\;
\QQ_{4v}^\II = \Fix \la \delta \ra = \QQ[\omega
  - \omega^{-1}] = \QQ[i \, \cs(r/4v)]
  = \QQ[i \, \sn(r/4v)] \, .}$

\smallskip
\noin Here, $r$ is any odd integer. These three
subfields all have index $2$ in $\QQ_{4v}$. In
other words, they have degree $v$ over $\QQ$.
Glancing back at the proof of Theorem 2.4, we
observe that $\beta$ and $\delta$ have no square
root in $\Aut(\QQ_{4v})$. So $\gamma$ belongs to
every non-trivial subgroup of $\Aut(\QQ_{4v})$
except for $\la \beta \ra$ and $\la \delta \ra$.
Therefore $\QQ_{2v}$ contains every strict
subfield of $\QQ_{4v}$ except for $\QQ_{4v}^\RR$
and $\QQ_{4v}^\II$. These observations yield a
complete description of the intermediate subfields
$\QQ \leq K \leq \QQ_{4v}$. (In particular, we see
that, letting $u$ be any power of $2$ with
$2 \leq u \leq v$, then there are precisely three
intermediate fields with degree $u$ over $\QQ$.
But there are four families of Roquette $2$-groups:
cyclic, dihedral, semidihedral, quaternion. This
already suggests that distinguishing between the
four families may be little awkward.)

In the following four lemmas, we still let
$v$ be a power of $2$ with $v \geq 2$. The
first one is similar to Lemma 5.1, and again,
it is obvious. We postpone discussion of the
vertex set.

\begin{lem}
Let $\psi$ be a faithful $\KK$-irrep of the
cyclic group $C_{2v}$. Then the order is
$v(\psi) = v$. The exponent is $n(\psi) =
2v$. The unique minimal splitting field for
$\psi_\QQ$ is the unique Fein field for
$\psi$, and it coincides with the vertex
field $\VV(\psi) = \QQ_{2v}$. The Schur index
is $m(\psi) = 1$. The Frobenius--Schur type
is $\Delta(\psi) = \CC$.
\end{lem}

\begin{lem}
Let $\psi$ be a faithful $\KK$-irrep of the
dihedral group $D_{8v}$. Then $v(\psi) = v$
and $n(\psi) = 4v$. The unique minimal
splitting field for $\psi_\QQ$ is the unique
Fein field for $\psi$, and it coincides with
the vertex field $\VV(\psi) = \QQ_{4v}^\RR$.
The Schur index is $m(\psi) = 1$. The
Frobenius--Schur type is $\Delta(\psi) = \RR$.
\end{lem}

\begin{proof}
Plainly, $v(\psi) = v$. Employing the standard
presentation, the group $D_{8v} = \la a, b \ra$
has a faithful irreducible matrix representation
$\psi$ given by $a \mapsto R(1/4q)$ and $b
\mapsto B$. By considering the matrix entries,
we see that $\psi$ is affordable over the field
$\QQ_{4v}^\RR$. Hence $\VV(\psi) \leq
\QQ_{4v}^\RR$. But we must have equality,
because the character value at $a$ is $\psi(a)
= 2 \cs(1/4v)$, which is a primitive element
of $\QQ_{4v}^\RR$. It is clear that $\psi$ has
all the specified properties. All of these
properties are invariant under Galois conjugation.
So, invoking Corollary 2.8, the properties hold
for any faithful $\KK D_{8v}$-irrep.
\end{proof}

\begin{lem}
Let $\psi$ be a faithful $\KK$-irrep of the
semidihedral group $\SD_{8v}$. Then $v(\psi) = v$
and $n(\psi) = 4v$. The unique minimal splitting
field for $\psi_\QQ$ is the unique Fein field for
$\psi$, and it coincides with the vertex field
$\VV(\psi) = \QQ_{4v}^\II$. The Schur index is
$m(\psi) = 1$. The Frobenius--Schur type is
$\Delta(\psi) = \CC$.
\end{lem}

\begin{proof}
The argument is similar to the proof of the
previous lemma. Note that $\SD_{8v}$ has a
faithful irreducible matrix representation
the standard generators $a$ and $d$ to the
matrices $I(1/4q)$ and $D$, respectively.
\end{proof}

\begin{lem}
Let $\psi$ be a faithful $\KK$-irrep of the
quaternion group $Q_{8v}$. Then $v(\psi) = v$
and $n(\psi) = 4v$. The vertex field is
$\VV(\psi) = \QQ_{4v}^\RR$. The unique Fein field
for $\psi$ is $\QQ_{4v}$. Two non-isomorphic
minimal splitting fields for $\psi_\QQ$ are
$\QQ_{4v}$ and $\QQ_{8v}^\II$. The Schur index
is $m(\psi) = 2$. The Frobenius--Schur type is
$\Delta(\psi) = \HH$.
\end{lem}

\begin{proof}
By Corollary 2.8 again, we may assume that
$\psi$ is the faithful $\CC Q_{4v}$-irrep such
that $\psi(a^r) = \omega^r + \omega^{-r} =
2 \cs(r/4v)$ for $r \in \ZZ$. There is a
matrix representation of $\psi$ such that
the standard generators $a$ and $x$ act as
$S(1/4q)$ and $X$, respectively. From the
character values, we see that $\VV(\psi) =
\QQ_{4v}^\RR$. Since the dihedral groups are
the only non-abelian finite groups with a
faithful representation on the Euclidian
plane, $\psi$ is not affordable over $\RR$.
(Alternatively, we can observe that
$\sum_{f \in A} \psi(f^2) = 0$ and $\psi(g^2)
= \psi(a^{2q}) = -2$ for $g \in Q_{4v} - \la
a \ra$, whence $\fs(\psi) = -1$.) Perforce,
$\psi$ is not affordable over $\QQ_{4v}^\RR$.
On the other hand, by considering the matrix
entries of $S(1/4q)$, we see that $\psi$ is
affordable over $\QQ_{4v}$. (Alternatively,
we can appeal to Brauer's Splitting Theorem.)
The quadratic extension $\QQ_{4v}$ of
$\QQ_{4v}^\RR$ must be a minimal splitting
field for $\psi$. It follows that $n(\psi) =
4v$ and $\QQ_{4v}$ is the unique Fein field of
$\psi$. These observations also imply that
$m(\psi) = |\QQ_{4v} : \QQ_{4v}^\RR| = 2$ and
$\Delta(\psi) = \HH$.

By direct calculation, it is easy to check that
$\psi$ has another matrix representation given by
$a \mapsto A_{1/8v}(1/4v)$ and $x \mapsto X$. The
field generated by the matrix entries of
$A_{1/8v}(1/4v)$ is $\QQ[\cs(1/4v), \sn(1/4v),
i \, \cs(1/8v), i \, \sn(1/8v)] = \QQ_{8v}^\II$,
and this must be a minimal splitting field
because it is a quadratic extension of the
vertex field. The two minimal splitting fields
that we have mentioned are non-isomorphic because
they are distinct subfields of the cyclotomic
field $\QQ_{8v}$, whose Galois group over $\QQ$
is abelian.
\end{proof}

We now discuss the vertex sets for the faithful
irreps of the Roquette $2$-groups. We continue
to assume that $v$ is a power of $2$ with
$v \geq 2$. Below, we shall find that, if
$p = 2$ and if $\psi$ is a $\KK G$-irrep with
order $v(\psi) = v$, then there are precisely
four possibilities for the genotype $\Typ(\psi)$,
namely $C_{2v}$, $D_{8v}$, $\SD_{8v}$, $Q_{8v}$.
To what extent can we distinguish between these
four possibilities by considering Galois actions
on the vertices? Recall, from Section 2, that the
vertex set $\Vtx(\psi)$ is a permutation set for
the Galois group of a sufficiently large Galois
extension of $\QQ$. The question will reduce to
a consideration of the Roquette $2$-groups. Let
$\psi_C$, $\psi_D$, $\psi_S$, $\psi_Q$ be
faithful $\KK$-irreps of $C_{2v}$, $D_{8v}$,
$\SD_{8v}$, $Q_{8v}$, respectively. Since
$n(\psi) = 2v$, we can regard $\Vtx(\psi_C)$ as
a permutation set for the Galois group
$\Aut(\QQ_{2v}) = \Gal(\QQ_{2v}/\QQ)$. More
generally, we can regard $\Vtx(\psi)$ as a
permutation set for $\Aut(\QQ_n)$ where $n$ is
any multiple of $2v$. Meanwhile, since $n(\psi_D)
= n(\psi_S) = n(\psi_Q) = 4v$, we can regard
$\Vtx(\psi_D)$ and $\Vtx(\psi_S)$ and
$\Vtx(\psi_Q)$ as permutation sets for
$\Aut(\QQ_{4v})$ and, more generally, as
permutation sets for $\Aut(\QQ_n)$ where $n$
is now any multiple of $4v$.
In view of these observations, we put
$n = 4v$. We regard all four vertex sets
$\Vtx(\psi_C)$, $\Vtx(\psi_D)$, $\Vtx(\psi_S)$,
$\Vtx(\psi_Q)$ as $\Aut(\QQ_{4v})$-sets. As we
noted in Section 2, all four of them are
transitive. Since $\Aut(\QQ_{4v})$ has size
$2v$ and since the four vertex sets all have
size $v$, the four point-stabilizer subgroups
all have size $2$. Of course, since
$\Aut(\QQ_{4v})$ is abelian, any transitive
$\Aut(\QQ_{4v})$-set has a unique
point-stabilizer subgroup.

\begin{lem}
With the notation above, the vertex sets
$\Vtx(\psi_C)$ and $\Vtx(\psi_D)$ and
$\Vtx(\psi_S)$ and $\Vtx(\psi_Q)$ are
transitive $\Aut(\QQ_{4v})$-sets, and the
point-stabilizer subgroups are $\la \gamma \ra$
and $\la \beta \ra$ and $\la \delta \ra$ and
$\la \beta \ra$, respectively.
\end{lem}

\begin{proof}
We apply the Fundamental Theorem of Galois
Theory to the Galois group $\Aut(\QQ_{4v})$
of the field extension $\QQ_{4v}/\QQ$. The
subgroups $\la \gamma \ra$ and $\la \beta \ra$
and $\la \delta \ra$ and $\la \beta \ra$ are
the centralizers of the subfields $\VV(\psi_C)
= \QQ_{2v}$ and $\VV(\psi_D) = \QQ_{4v}^\RR$
and $\VV(\psi_S) = \QQ_{4v}^\II$ and
$\VV(\psi_Q) = \QQ_{4v}^\RR$, respectively.
\end{proof}

To begin the slow movement, let us recall some
obligations from Section 2. There, we listed
some invariants of a $\KK G$-irrep $\psi$, and
we stated that they are genetic invariants. We
also stated that there exists a unique Fein field
for $\psi$. We indicated that we would recover
Schilling's Theorem. We stated that the vertex set
$\Vtx(\psi)$ is the maximum field that embeds
in every splitting field for $\psi_\QQ$. In the
next few results, we shall prove those assertions.

\begin{thm}
Let $\psi$ be a $\KK G$-irrep. Let $\LL/\JJ$ be
a characteristic zero field extension. Then the
$\LL/\JJ$-relative Schur index $m_\JJ^\LL(\psi)$,
the $\LL/\JJ$-relative order $v_\JJ^\LL(\psi)$,
the $\JJ$-algebra isomorphism class of the
endomorphism ring $\End_{\JJ G}(\psi)$, the class
of minimal splitting fields for $\psi_\JJ$ and
the $\JJ$-relative vertex field $\VV_\JJ(\psi)$
are genetic invariants of $\psi$. In particular,
$m_\JJ(\psi)$, $v_\JJ(\psi)$, $m(\psi)$,
$v(\psi)$ and $\VV(\psi)$ are genetic invariants.
There exists a unique Fein field $\Fein(\psi)$.
Furthermore, $\Fein(\psi)$ is a genetic invariant.
The $\Aut(\VV(\psi))$-set isomorphism class of
vertex set $\Vtx(\psi)$, the Frobenius--Schur
type $\Delta(\psi)$ and the genotype $\Typ(\psi)$
are genetic invariants.
\end{thm}

\begin{proof}
Obviously, $\Typ(\psi)$ is a global invariant.
By the Field-Changing Theorem 3.5, the genetic
subquotients for $\psi$ coincide with the
genetic subquotients for $\psi_\QQ$. Therefore
$\Typ(\psi)$ is a quasiconjugacy invariant.
Given a subgroup $L \leq G$ and a $\KK L$-irrep
$\theta$ from which $\psi$ is tightly induced,
then, by Remark 3.3, every genetic subquotient
for $\theta$ is a genetic subquotient for
$\psi$. Therefore $\Typ(\psi)$ is a tight
induction invariant. We have shown that
$\Typ(\psi)$ is a genetic invariant.

Let us write $[\End_\QQ]$ to denote the
isomorphism class of the ring $\End_\QQ =
\End_{\QQ G}(\psi_\QQ)$. As we already noted
in Section 2, $[\End_\QQ]$ is a quasiconjugacy
global invariant. By the definition of shallow
induction, $[\End_\QQ]$ is a tight induction
invariant. So $[\End_\QQ]$ is a genetic invariant.

The $\JJ$-algebra $\JJ \otimes_\QQ \End_\QQ$
is isomorphic to a direct sum of $v_\JJ(\psi)$
copies of the ring of $m_\JJ(\psi) \times
m_\JJ(\psi)$ matrices over the $\JJ$-algebra
$\End_\JJ = \End_{\JJ G}(\JJ \psi_\QQ)$. So
$[\End_\QQ]$ determines $v_\JJ(\psi)$ and
$m_\JJ(\psi)$. Furthermore, $[\End_\QQ]$
determines the isomorphism class of $\End_\JJ$
and, in particular, the isomorphism class of
$\End_\RR = \Delta_\psi$. The $\LL$-algebra
$\LL \otimes_\JJ \End_\JJ$ is isomorphic to a
direct sum of $v_\JJ^\LL(\psi)$ copies of the
ring of $m_\JJ^\LL(\psi) \times m_\JJ^\LL(\psi)$
matrices over $\End_\LL$. So $[\End_\QQ]$
determines $v_\JJ^\LL(\psi)$ and
$m_\JJ^\LL(\psi)$. We have $\VV_\JJ(\psi)
\cong Z(\End_\JJ)$, so $[\End_\QQ]$ determines
$\VV_\JJ(\psi)$. The splitting fields
for $\psi_\JJ$ are precisely the splitting
fields for $\End_\JJ$. So $[\End_\QQ]$
determines the class of minimal splitting
fields for $\psi_\JJ$. It follows that
$[\End_\QQ]$ determines the $n(\psi)$ and
the class of Fein fields for $\psi$. As
$\Aut(\VV(\psi))$-sets, $\Vtx(\psi)$ is
isomorphic to the set of Wedderburn
components of the semisimple ring
$\VV(\psi) \otimes_\QQ \End_\QQ$. So
$[\End_\QQ]$ determines $\Vtx(\psi)$.
With the exception of $\Typ(\psi)$, all
the specified invariants are thus determined
by the genetic invariant $[\End_\QQ]$, hence
they are genetic invariants. By the way, an
easier way to see the genetic invariance of
$\Vtx(\psi)$ is to observe that the
quasiconjugacy global invariance is obvious,
while the tight induction invariance is
immediate from condition (b) in the Narrow
Lemma 3.2.

It remains only to demonstrate the existence
and uniqueness of the Fein field. Let $H/K$
be a genetic subquotient of $\psi$ and let
$\phi$ be the germ of $\psi$ at $H/K$. Since
the class of Fein fields is a genetic
invariant, the Fein fields of $\psi$ coincide
with the Fein fields of $\phi$. Replacing
$\psi$ with $\phi$, we reduce to the case
where $G$ is Roquette and $\psi$ is faithful.
If $G = C_1$ or $G = C_2$, then the unique
$\QQ$-relative Fein field is $\Fein(\psi) =
\QQ$. When $|G| \geq 3$, the existence and
uniqueness of $\Fein(\psi)$ was already shown
in Lemmas 5.1, 5.2, 5.3, 5.4, 5.5.
\end{proof}

The argument in the last paragraph of the proof
can be abstracted in the form of the following
remark.

\begin{rem}
Let $\psi$ be a $\KK G$-irrep. Let $\JJ$ be a
field with characteristic zero, and let $\phi$
be a faithful $\JJ \Typ(\psi)$-irrep. Let $\cI$
be an invariant defined on characteristic zero
irreps of finite $p$-groups. If $\cI$ is a
genetic invariant, then $\cI(\psi) = \cI(\phi)$.
\end{rem}

Thus, any genetic invariant is determined by the
genotype. Conversely, the latest theorem tells
us that the genotype is a genetic invariant.
The following corollary is a restatement of
those two conclusions.

\begin{cor}
For irreps of finite $p$-groups over a field with
characteristic zero, the genetic invariants are
precisely the isomorphism invariants of the
genotype.
\end{cor}

The field $\Fein(\psi)$ need not be the only
minimal splitting field for $\psi$ contained
in $\QQ_{|G|}$. Lemma 5.5 shows that every
quaternion $2$-group is a counter-example.
For a complex irrep $\chi$ of an arbitrary
finite group $F$, the splitting field
$\QQ_{|G|}$ need not contain a minimal
splitting field for $\chi$. Fein \cite{Fei74}
gave a counter-example where $|F|$ has
precisely three prime factors. In the same
paper, he showed that, if $|F|$ has precisely
two prime factors and if $\chi$ has Schur
index $m(\chi) \geq 3$, then $\QQ_n$ contains
a minimal splitting field, where $n$ is the
exponent of $G$.
However, as we are about to show, the condition
$m(\chi) \geq 3$ always fails when $F$ is a
$p$-group. The line of argument by which we
arrive at the following celebrated result is
due to Roquette \cite{Roq58}, but it is worth
assimilating into our account because it is
a paradigm for the genetic reduction technique.

\begin{thm}
{\rm (Schilling's Theorem)} Given a
$\KK G$-irrep $\psi$ and a field extension
$\LL/\JJ$ with characteristic zero, then
$m_\JJ^\LL(\psi) \leq 2$. If $m_\JJ^\LL(\psi)
= 2$ then $\Delta(\psi) = \HH$. If
$\Delta(\psi) = \HH$, then $m(\psi) = 2$.
\end{thm}

\begin{proof}
By the latest theorem and the subsequent
remark, we may assume that $G$ is Roquette
and that $\psi$ is faithful. From the definition
of the relative Schur index, $m_\QQ^\LL(\psi) =
m_\QQ^\JJ(\psi) m_\JJ^\LL(\psi)$. So
$m_\JJ^\LL(\psi) \leq m_\QQ^\LL(\psi) \leq
m(\psi)$. It suffices to show that $m(\psi)
\leq 2$ with equality if and only if
$\Delta(\psi) = \HH$. Applying Lemmas 5.1,
5.2, 5.3, 5.4, and attending separately to
the degenerate case $|G| \leq 2$, we deduce
that if $G$ is cyclic, dihedral or semidihedral,
then $m(\psi) = 1$ and $\Delta(\psi) \neq \HH$.
If $G$ is quaternion then, by Lemma 5.5,
$m(\psi) = 2$ and $\Delta(\psi) = \HH$.
\end{proof}

The next result is probably of no technical
interest, but it does at least indicate why we
call $\VV(\psi)$ the vertex field. However, the
analogous assertion can fail for the relative
vertex field: if $G$ is a quaternion $2$-group
then $\VV_\RR(\psi) = \RR$, but the unique
minimal splitting field for $\psi$ is $\CC$.

\begin{pro}
Let $\psi$ be a $\KK G$-irrep. Partially
ordering isomorphism classes of fields by
embedding, then (the isomorphism class of)
$\VV(\psi)$ is the unique maximal field 
that embeds in all the minimal splitting
fields of $\psi$.
\end{pro}

\begin{proof}
By the latest theorem and remark, we may
assume that $G$ is Roquette and that $\psi$
is faithful. The assertion is now clear
from Lemmas 5.1, 5.2, 5.3, 5.4, 5.5.
\end{proof}

We shall end this movement by showing that
the genotype $\Typ(\psi)$ of a non-trivial
$\KK G$-irrep $\psi$ is determined by the ring
$\End_{\QQ G}(\psi_\QQ)$, and the genotype
is also determined by the class of minimal
splitting fields for $\psi_\QQ$. Note that,
aside from the genotype itself, none of the
genetic invariants listed in Theorem 5.7
can be used to distinguish between genotype
$C_1$ and genotype $C_2$. But those two
genotypes can be distinguished very easily:
a Frobenius reciprocity argument shows that
$\Typ(\psi) = C_1$ if and only if $\psi$ is
the trivial $\KK G$-irrep.

The following corollary relies on the Genotype
Theorem 1.1. Indeed, it relies on Theorem 5.7.
Although we did not mention the Genotype Theorem
in the above proof of Theorem 5.7, we implicitly
used the Genotype Theorem because our argument
involved $\Typ(\psi)$, whose existence and
uniqueness is guaranteed by the Genotype Theorem.
However, the reasoning that has led us to the
following corollary makes essential use only of
the existence, not the uniqueness. The existence
of $\Typ(\psi)$ is captured in Theorem 4.1,
which was proved by a direct argument in
Section 4. So, with the following corollary, we
complete a direct proof of the Genotype Theorem,
avoiding the reduction to the special case
$\KK = \QQ$ or $\KK = \CC$.

\begin{cor}
Let $\psi$ be a non-trivial $\KK G$-irrep. Let
$\psi'$ be a non-trivial $\KK G'$-irrep, where
$G'$ is a finite $p'$-group and $p'$ is a prime.
Then the following conditions are equivalent.

\smallskip
\noin {\bf (a)} $\Typ(\psi) = \Typ(\psi')$.

\smallskip
\noin {\bf (b)} $\End_{\QQ G}(\psi_\QQ) \cong
\End_{\QQ G'}(\psi'_\QQ)$.

\smallskip
\noin {\bf (c)} The minimal splitting fields for
$\psi$ coincide with the minimal splitting fields
for $\psi'$.
\end{cor}

\begin{proof}
Since the endomorphism ring
$\End_{\QQ G}(\psi_\QQ)$ is a genetic invariant,
it is isomorphic to the endomorphism ring of the
faithful rational irrep of $\Typ(\psi)$. So (a)
implies (b). The minimal splitting fields for
$\psi$ are precisely the minimal splitting fields
for $\End_{\QQ G}(\psi_\QQ)$. So (b) implies (c).
Suppose that (c) holds. To deduce (a), the latest
theorem and remark allow us to assume that $G$
and $G'$ are Roquette. If $\QQ$ is a splitting
field for $\psi$ and $\psi'$, then $\Typ(\psi)
= C_2 = \Typ(\psi')$. Otherwise, the equality
of the two genotypes follows from the first
five lemmas in this section.
\end{proof}

Finally, we are ready to present the synthesis
of the material in the previous two movements.
Corollary 5.12 is unlikely to be of much use
towards evaluating the genotype of an explicitly
given irrep. The following theorem can be
applied first to evaluate the genotype from
more easily ascertained genetic invariants. The
genotype having been evaluated, the theorem can
be applied again to evaluate other genetic
invariants. (The above proof of Schilling's
Theorem can be cast in that form. Anyway, we
are not suggesting that anyone would actually
wish to evaluate genetic invariants for
numerically specified irreps. It can be argued
that, in pure mathematics no less than in the
other sciences, a sufficient criterion for
meaningful content should be only that the
material could be applied efficiently and
effectively to some natural class of problems;
without requiring that there be any demand for
the solutions to those problems.)

\begin{thm}
Let $\psi$ be a $\KK G$-irrep and let
$v = v(\psi)$. First suppose that $v = 1$.
Then precisely one of the following three
conditions holds.

\smallskip
\noin {\bf (a)} $\Typ(\psi) = C_1$ and $\psi$
is the trivial $\KK G$-irrep.

\smallskip
\noin {\bf (b)} $\Typ(\psi) = C_2$ and $\psi$ is
non-trivial, affordable over $\QQ$ and absolutely
irreducible. In particular, the Schur index is
$m(\psi) = 1$ and the Frobenius--Schur
type is $\Delta(\psi) = \RR$.

\smallskip
\noin {\bf (c)} $\Typ(\psi) = Q_8$ and $m(\psi)
= 2$ and $\Delta(\psi) = \HH$.

\smallskip
\noin Now suppose that $p$ is odd and $v \neq 1$.
Then the exponent $n = n(\psi)$ is a power of
$p$ and $v = n(1-1/p)$. Also, $m(\psi) = 1$ and
$\Delta(\psi) = \CC$. The unique minimal splitting
field for $\psi$ is the field $\Fein(\psi) =
\VV(\psi) = \QQ_n$. The vertex set $\Vtx(\psi)$
is free and transitive as an $\Aut(\QQ_n)$-set.

\smallskip
\noin Now suppose that $p = 2$ and
$v \neq 1$. Then $v$ is a power of $2$ and
precisely one of the following conditions holds.

\smallskip
\noin {\bf (C)} $\Typ(\psi) = C_{2v}$ and $n(\psi)
= 2v$ and $m(\psi) = 1$ and $\Delta(\psi) = \RR$.
The unique minimal splitting field for $\psi$ is
the field $\Fein(\psi) = \VV(\psi) = \QQ_{2v}$.
As a transitive $\Aut(\QQ_{4v})$-set, $\Vtx(\psi)$
has point-stabilizer subgroup $\la \gamma \ra$.

\smallskip
\noin {\bf (D)} $\Typ(\psi) = D_{8v}$ and $n(\psi)
= 4v$ and $m(\psi) = 1$ and $\Delta(\psi) = \RR$.
The unique minimal splitting field for $\psi$ is
the field $\Fein(\psi) = \VV(\psi) = \QQ^\RR_{4v}$.
As a transitive $\Aut(\QQ_{4v})$-set, $\Vtx(\psi)$
has point-stabilizer subgroup $\la \beta \ra$.

\smallskip
\noin {\bf (S)} $\Typ(\psi) = \SD_{8v}$ and $n(\psi)
= 4v$ and $m(\psi) = 1$ and $\Delta(\psi) = \CC$.
The unique minimal splitting field for $\psi$ is
the field $\Fein(\psi) = \VV(\psi) = \QQ^\II_{4v}$.
As a transitive $\Aut(\QQ_{4v})$-set, $\Vtx(\psi)$
has point-stabilizer subgroup $\la \delta \ra$.

\smallskip
\noin {\bf (Q)} $\Typ(\psi) = Q_{8v}$ and $n(\psi)
= 4v$ and $m(\psi) = 2$ and $\Delta(\psi) = \HH$.
Two non-isomorphic splitting fields for $\psi$ are
$\QQ_{4v}$ and $\QQ_{8v}^\II$. Also, $\Fein(\psi)
= \QQ_{4v}$ and $\VV(\psi) = \QQ^\RR_{4v}$. As a
transitive $\Aut(\QQ_{4v})$-set, $\Vtx(\psi)$
has point-stabilizer subgroup $\la \beta \ra$.
\end{thm}

\begin{proof}
The case $|\Typ(\psi)| \leq 2$ is easy. The
rest follows from the first eight results
in this section.
\end{proof}

When $p = 2 \leq v$, one routine for calculating
the genotype is to find the values of $v$ and
$\fs(\psi)$. If $\fs(\psi) = 0$ then the
possible genotypes $C_{2v}$ and $\SD_{8v}$
can be distinguished using the fact that,
in the former case, the involution fixing
$\psi$ is $\gamma$ while, in the latter
case, the involution fixing $\psi$ is $\delta$.
Another routine is to evaluate $\VV(\psi)$ and,
if necessary, $\fs(\psi)$.

To reinforce the point, let us return, once
again, to the real irrep $\chi$ of the
group $\DD = \DD_{2^n} = \DD_{16u}$. We
evaluated $\Typ(\psi)$ already at the end of
Section 4, but let us now do it more swiftly.
Using part (3) of Lemma 2.5, we see that
$v(\psi) = u/2$. By considering the partition
$$\DD = A \cup (D_{8u} - A) \cup (\Mod_{8u}
  - A) \cup (\SD_{8u} - A)$$
we see that $\fs(\chi) = 1$ and $\Delta(\chi)
= \RR$. We recover the conclusion that
$\Typ(\chi) = D_{2^{n-2}}$ if $n \geq 6$
while $\Typ(\chi) = C_2$ if $n = 5$.

For another example, still with $n \geq 5$,
suppose that $G$ is the smash product
$C_4 * D_{2^{n-1}}$, which has order $2^n$.
Each faithful complex irrep of the subgroup
$1 * D_{2^{n-1}}$ extends to two complex
conjugate irreps of $G$. So there are
precisely $2^{n-3}$ faithful $\CC G$-irreps,
and they comprise a single Galois conjugacy
class. Let $\psi$ be a faithful $\KK G$-irrep.
Then $v(\psi) = 2^{n-3}$. The two generators
of $C_4 * 1$ act on $\psi_\CC$ as scalar
multiplication by $\pm i$, and the faithful
complex irreps of $1 * D_{2^{n-1}}$ have
vertex field $\QQ_{2^{n-2}}^\RR$, hence
$$\VV(\psi) = \QQ_{2^{n-2}}^\RR[i] =
\QQ_{2^{n-2}} = \QQ_{2v(\psi)} \, .$$
Therefore, $\Typ(\psi) = C_{2^{n-2}}$. At
the beginning of Section 2, we noted that
the unique faithful $\QQ C_4 * D_{16}$-irrep
is induced from a $C_4$ subquotient and also
from a $C_8$ subquotient; we have made some
progress since then, and we can now announce
that, actually, the genotype is $C_8$.

\section{Counting Galois conjugacy classes
of irreps}

We shall be correlating some results of
tom Dieck \cite[III.5.9]{Die87} and Bouc
\cite[8.5, 8.7]{Bou*}. They were concerned
with the cases where $\KK$ is $\QQ$ or $\RR$
or $\CC$. Our generalizations to the case of
arbitrary $\KK$ are slender, although we should
point out that the construction of $\ooR(\KK G)$,
below, does rely on the notion of Galois
conjugacy that was established in Section 2.
What is of more interest is that we shall be
providing quicker proofs, making use of the
fact that the genetic theory applies directly
to the general case.

\begin{rem}
Let $R$ be a Roquette $p$-group, and let $k_R(G)$
denote the number of Galois conjugacy classes of
$\KK G$-irreps with genotype $R$. Then $k_R(G)$
is independent of $\KK$.
\end{rem}

The remark is immediate from the Field-Changing
Theorem 3.5. We mention that $k_R(G)$ is a global
quasiconjugacy invariant. Letting $k_*(G) = \sum_R
k_R(G)$, where $R$ runs over all the Roquette
$p$-groups, then $k_*(G)$ is the number of Galois
conjugacy classes of $\KK G$-irreps, in other
words, the number of $\QQ G$-irreps, we mean to
say, the number of conjugacy classes of cyclic
subgroups of $G$.

Recall that a {\bf superclass} function for $G$
is a $\ZZ$-valued function $f$ on the set of
subgroups of $G$ such that $f$ is constant on
each conjugacy class of subgroups. The {\bf
superclass ring} of $G$, denoted $C(G)$, is
understood to be the additive group consisting
of the superclass functions on $G$. The {\bf
representation ring} $R(\KK G)$, also called
the {\bf character ring}, is understood
to be the group of virtual $\KK G$-reps (the
universal abelian group associated with the
semigroup of $\KK G$-reps). Of course, in many
well-known applications, $C(G)$ and $R(\KK G)$
are assigned all sorts of further structures (in
particular, they are rings) but those further
structures are irrelevant to our concerns. We
shall be regarding $C(G)$ and $R(\KK G)$ merely
as free abelian groups.

We define the {\bf tom Dieck map}
$$\Die_G^\KK \: : \: R(\KK G) \rightarrow C(G)$$
to be the linear map such that, given a
$\KK G$-rep $\xi$ and a subgroup $H \leq G$,
then the value of $\Die_G^\KK(\xi)$ at $H$ is
equal to the multiplicity of the trivial
$\KK H$-irrep in the restriction $\res_H^G(\xi)$.
Thus, treating $\xi$ as a $\KK G$-module and
writing its $H$-fixed subspace as $\xi^H$, we
have $\Die_G^\KK(\xi)(H) = \dim_\KK(\xi^H)$.
Let $I(\KK G)$ be the subgroup of $R(\KK G)$
generated by the elements having the form
$\xi - \xi'$, where $\xi$ and $\xi'$ are Galois
conjugate $\KK G$-reps. We write $\ooxi = \xi +
I(\KK G)$ as an element of the quotient group
$\ooR(\KK G) = R(\KK G)/I(\KK G)$. Since
$\Die_G^\KK$ annihilates $I(\KK G)$, we can
define another {\bf tom Dieck map}
$$\ooDie_G^\KK \: : \: \ooR(\KK G)
  \rightarrow C(G)$$
such that $\ooDie_G^\KK(\ooxi) = \Die_G^\KK(\xi)$.

Letting $\psi_1$, $...$, $\psi_r$ be a set of
representatives of the Galois conjugacy classes
of $\KK G$-irreps, then $\{ \oopsi_1, ...,
\oopsi_r \}$ is a $\ZZ$-basis for $\ooR(\KK G)$,
and $\{ \overline{(\psi_1)_\QQ}, ...,
\overline{(\psi_r)_\QQ} \}$ is a $\ZZ$-basis for
$\ooR(\QQ G)$. We mention that there is a
isomorphism $\ooR(\KK G) \rightarrow \ooR(\QQ G)$
such that $\oopsi_j \mapsto \overline{(\psi_j)_\QQ}$.
But the isomorphism does not commute with the tom
Dieck maps $\ooDie_G^\KK$ and $\ooDie_G^\QQ$.

The next two results are due to tom Dieck
\cite[III.5.9, III.5.17]{Die87}. Our slight
embellishment is to extend to the case where $\KK$
is arbitrary. Although the proofs presented below
are different from tom Dieck's, the ideas are
implicit in \cite[III.5]{Die87}.

\begin{thm}
{\rm (tom Dieck)} The tom Dieck map
$\ooDie_G^\KK$ is injective, and its image is a
free abelian group whose rank is $k_*(G)$.
\end{thm}

\begin{proof}
By Theorem 2.2, we may assume that $\KK = \CC$.
Suppose that $\ooDie_G^\CC$ is not injective. Then
$$a_1 \dim_\CC(\psi_1^H) + ... +
  a_r \dim_\CC(\psi_r^H) =
  \Die_G^\CC(a_1 \psi_1 + ... a_r \psi_r) = 0$$
where $\psi_1$, $...$, $\psi_r$ are mutually
Galois non-conjugate $\CC G$-irreps and each
$a_j$ is a non-zero integer. First consider the
case where some $\psi_j$ is non-faithful. Then,
without loss of generality, there is an integer
$s \leq r$ and a non-trivial normal subgroup $K$
of $G$ such that the kernels of $\psi_1$, $...$,
$\psi_s$ all contain $K$ while the kernels of
$\psi_{s+1}$, $...$, $\psi_r$ do not contain
$K$. When $s < j \leq r$, Clifford theory
yields $\dim_\CC(\psi_j^K) = 0$ and, perforce,
$\dim_\CC(\psi_j^H) = 0$ for all intermediate
subgroups $K \leq H \leq G$. Replacing $G$ with
$G/K$, we obtain a contradiction by insisting that
$|G|$ was minimal. Now consider the case where all
the $\psi_j$ are faithful and $G$ is not Roquette.
Let $E$ and $T$ be as in Lemma 4.2. The proof of
that lemma shows that each $\res_T^G(\psi_j) =
\psi_{j,1} + ... + \psi_{j,p}$ as a sum of
mutually Galois non-conjugate $\CC T$-irreps.
But $\psi_j = \ind_T^G(\psi_{j,i})$ and it follows
that, as $j$ and $i$ run over the ranges $1 \leq j
\leq r$ and $1 \leq i \leq p$, the $\CC T$-irreps
$\psi_{j,i}$ are mutually Galois non-conjugate.
Replacing $G$ with $T$, we again obtain a
contradiction by induction on $|G|$. We have
reduced to the case where $G$ is Roquette and
$\psi_1$, $...$, $\psi_r$ are faithful. By
Corollary 2.8, $r = 1$. Absurdly, we deduce that
$\dim_\CC(\psi_1^H) = 0$ for all subgroups
$H \leq G$.
\end{proof}

We write $\mod_2$ to indicate reduction modulo
$2$: for a free abelian group $A$, we write
$\mod_2(A) = (\ZZ/2) \otimes_\ZZ A$; we write
$\mod_2$ for canonical epimorphism $A \rightarrow
\mod_2(A)$. The composite maps $\mod_2 \,
\Die_G^\KK : R(\KK G) \rightarrow \mod_2(C(G))$
and $\mod_2 \, \ooDie_G^\KK : \ooR(\KK G)
\rightarrow \mod_2(C(G))$ are still called
{\bf tom Dieck maps}.

\begin{thm}
{\rm (tom Dieck)} Suppose that $p=2$. Then the image
$\Im(\mod_2 \, \ooDie_G^\KK)$ is an elementary
abelian $2$-group whose rank is the number of Galois
conjugacy classes of $\KK G$-abirreps with cyclic,
dihedral or semidihedral genotype. These are the
abirreps with Frobenius--Schur type $\RR$ or $\CC$.
\end{thm}

\begin{proof}
The rider will follow from the main part
together with Theorem 5.13. For any
$\KK G$-irrep $\psi$, the order $v(\psi)$ and
the Schur multiplier $m(\psi)$ are powers of
$2$. Therefore $\mod_2 \, \Die_G^\KK$ annihilates
any $\KK G$-irrep that is not absolutely
irreducible. So, if the required conclusion holds
for the algebraic closure of $\KK$, then it will
hold for $\KK$. Therefore, we may assume that
$\KK = \CC$.

Let $\psi_1$, $...$, $\psi_r$ be mutually Galois
non-conjugate $\CC G$-irreps with non-quaternion
genotypes. We claim that the linearly independent
elements $\oopsi_1$, $...$, $\oopsi_r$ of $\ooR(\CC
G)$ are sent by $\mod_2 \, \ooDie_G^\KK$ to linearly
independent elements of $\mod_2(C(G))$. Assuming
otherwise, and taking $r$ to be minimal, then
$$\ooDie_G^\KK(\oopsi_1) + ...
  + \Die_G^\KK(\oopsi_r) = 0.$$
That is to say, the dimension of $\psi_1^H \oplus
... \oplus \psi_r^H$ is even for all $H \leq G$.
Arguing as in the proof of the previous theorem,
we reduce to the case where $r = 1$ and $\psi_1$
is faithful and $G$ is a non-quaternion Roquette
$2$-group. Write $\psi = \psi_1$. If $G$ is
cyclic, then the subspace fixed by the trivial
group has dimension $\dim_\CC(\psi^1) = 1$,
and this is a contradiction. Supposing that
$G$ is dihedral, and writing $G = \la a, b \ra$
as in the standard presentation, then
$\dim_\CC(\psi^{\la b \ra}) = 1$, which is a
contradiction. In the semidihedral case we
obtain a contradiction using the standard
generator $d$ instead of $b$. Any which way,
we arrive at a contradiction. The claim is
established.

Now let $G$ be any $2$-group and let $\psi$ be
a $\CC G$-irrep having genotype $Q_{2^m}$ with
$m \geq 3$. The argument will be complete when
we have shown that $\ooDie_G^\CC(\oopsi) = 0$.
In other words, it remains only to show that
$\psi^H$ has even dimension for all $H \leq G$.
Consider the case where $G = Q_{2^m}$. The
center $Z = Z(G)$ has order $2$, and it is
contained in every non-trivial subgroup of
$G$. There are no non-zero $Z$-fixed points
in $\psi$, hence $\dim_\CC(\psi^H)$ is $2$
or $0$, depending on whether $H$ is trivial
or non-trivial, respectively. Either way,
$\dim_\CC(\psi^H)$ is even. Return now to
the case where $G$ is arbitrary. Let $L/M$
be a genetic subquotient for $\psi$ and let
$\phi$ be the germ of $\psi$ at $L/M$. Of
course, $L/M \cong Q_{2^m}$. By Frobenius
reciprocity and Mackey decomposition,
$$\dim_\CC(\psi^H) = \la 1 | \res_H^G \,
  \ind_L^G(\psi) \ra = \sum_{LgH \subseteq G}
  \la 1 | \res_{L \cap {}^g H}^L(\psi) \ra
  = \sum_{LgH \subseteq G}
  \dim_\CC(\phi^{L \cap {}^g H}).$$
From our discussion of the case $G = Q_{2^m}$,
we see that each term of the sum is
$0$ or $2$. So $\dim_\CC(\psi^H)$ is even.
\end{proof}

We now explain how, in the special case where
$\KK$ is a subfield of $\RR$, the group
$\mod_2(C(G))$ can be replaced with the unit
group $B(G)^\times$ of the Burnside ring $B(G)$.
This will lead to a new proof of a theorem of
Bouc. We sketch the prerequisite constructions,
referring to Yoshida \cite{Yos90} and
Yal\c{c}{\i}n \cite{Yal*} for details (we
employ much the same notation as the latter).
Given a $G$-set $X$, we write $[X]$ to denote
the isomorphism class of $X$ as an element of
$B(G)$. Recall that the species of $B(G)$
have the form $s_H : B(G) \ni [X] \mapsto
|X^H| \in \ZZ$, and two species $s_H$ and
$s_{H'}$ are equal if and only if $H$
and $H'$ are $G$-conjugate. Furthermore,
an element $x$ of $B(G)$ is determined by
the superclass function $H \mapsto s_H(x)$.
Note that $x$ is a unit if and only if $s_H(x)
= \pm 1$ for all $H$. Therefore, $B(G)^\times$
is an elementary abelian $2$-group.

Given an integer $c$ (possibly given only up
to congruence modulo $2$), we write
$\ppar(c) = (-1)^c$. Another theorem of tom
Dieck \cite[5.5.9]{Die79} asserts that, given
any $\RR G$-rep $\xi$, then there is an element
$\die_G^\RR(\xi) \in B(G)^\times$ such that
$s_H(\die_G^\RR(\xi)) = \ppar(\dim_\RR(\xi^H))$.
Hence, when $\KK$ is a subfield of $\RR$, there
is a linear map
$$\die_G^\KK \: : \: R(\KK G) \rightarrow
  B(G)^\times$$
such that, again, $s_H(\die_G^\KK(\xi)) =
\ppar(\dim_\KK(\xi^H))$. Since $\die_G^\KK$
annihilates the ideal $I(\KK G)$, it gives
rise to a linear map
$$\oodie_G^\KK \: : \: \ooR(\KK G)
  \rightarrow B(G)^\times \, .$$
The maps $\die_G^\KK$ and $\oodie_G^\KK$ are
called {\bf tom Dieck maps} because, as we
shall see in a moment, they are essentially
the same as the maps $\Die_G^\KK$ and
$\ooDie_G^\KK$. But let us emphasize that
$\die_G^\KK$ and $\oodie_G^\KK$ are defined
only when $\KK \leq \RR$.

The following result can be quickly obtained
from the special case $\KK = \RR$, which is
equivalent to Bouc \cite[8.5]{Bou*}. See
Corollary 6.6. We give a different proof.

\begin{thm}
{\rm (Bouc)}
Suppose that $\KK \leq \RR$. Then the image
$\Im(\oodie_G^\KK : \ooR(\KK G) \rightarrow
B(G)^\times)$ is an elementary abelian $2$-group
whose rank is the number of Galois conjugacy
classes of absolutely irreducible $\KK G$-irreps.
\end{thm}

\begin{proof}
When $p$ is odd, the assertion is virtually
trivial. Indeed, for any group of odd
order, the trivial irrep is the unique
real abirrep, and meanwhile, tom Dieck
\cite[1.5.1]{Die79} asserts that, for any
group of odd order, the unit group of the
Burnside ring is isomorphic to $C_2$.

Suppose that $p = 2$. The tom Dieck maps $\mod_2
\, \ooDie_G^\KK : \ooR(\KK G) \rightarrow
\mod_2(C(G))$ and $\mod_2 \, \oodie_G^\KK :
\ooR(\KK G) \rightarrow B(G)^\times$ commute
with the monomorphism of elementary abelian
$2$-groups $B(G)^\times \rightarrow \mod_2(C(G))$
which sends a unit $u$ to the superclass function
$f$ such that $s_H(u) = \ppar(f(H))$. By Theorem
6.3, the rank $\rk(\Im(\mod_2 \, \ooDie_G^\KK))
= \rk(\Im(\mod_2 \, \oodie_G^\KK))$ is the number
of Galois conjugacy classes of $\KK G$-abirreps
with Frobenius--Schur type $\RR$ or $\CC$. The
hypothesis on $\KK$ ensures that all the
$\KK G$-abirreps have type $\RR$.
\end{proof}

In order to recover the two special cases
stated in Bouc \cite[8.5, 8.7]{Bou*}, we need
the following result, which was obtained by
Tornehave \cite{Tor84} using topological
methods. Another proof was given by
Yal\c{c}{\i}n \cite{Yal*} using algebraic
methods. Actually, Yal\c{c}{\i}n reduced
to the case of a Roquette $2$-group, and
his paper is another application of the
genetic reduction technique. Note that
the case of odd $p$ is trivial.

\begin{thm}
{\rm (Tornehave)}
The tom Dieck map $\oodie_G^\RR : \ooR(\RR G)
\rightarrow B(G)^\times$ is surjective.
\end{thm}

From the latest two results, we recover the
following corollary, which was obtained by
Bouc \cite[8.5]{Bou*} using a filtration of
$B(\dash)^\times$ as a biset functor.

\begin{cor}
{\rm (Bouc)}
The number of Galois conjugacy classes of
$\RR G$-abirreps is the rank of $B(G)^\times$
as an elementary abelian $2$-group.
\end{cor}

It is worth pointing out how the general
version of Bouc's Theorem 6.4 covers another of
his results, \cite[8.7]{Bou*}, which we shall
present as the next corollary. Let $\lin_G$
denote the linearization map $B(G) \rightarrow
R(\KK G)$. We mean to say, $\lin_G[X]$
is the permutation $\KK G$-rep associated
with the $G$-set $X$. Let $\exp_G$ denote
the exponential map $B(G) \rightarrow
B(G)^\times$. We mean to say, $s_H(\exp_G[X])$
is the number of $H$-orbits in $X$. To see that
this condition really does determine a unit
$\exp_G[X] \in B(G)^\times$, observe that we
have a commutative triangle $\exp_G =
\die_G^\KK \, \lin_G$.

\begin{cor}
{\rm (Bouc)}
Suppose that $p = 2$. Then the number of
$\QQ G$-abirreps is $\rk(\Im(\exp_G))$.
Furthermore, $\exp_G$ is surjective
if and only if there are no $\QQ G$-irreps
with dihedral genotype.
\end{cor}

\begin{proof}
The first part follows from Theorem 6.4 together
with the Ritter--Segal Theorem, which asserts
that the linearization map $\lin_G^\QQ : B(G)
\rightarrow R(\QQ G)$ is surjective.
By Theorem 5.13, the $\QQ G$-abirreps are
precisely the $\QQ G$-irreps whose genotype
is $C_1$ or $C_2$. Also, the $\RR G$-abirreps
are precisely the $\RR G$-irreps whose genotype
is $C_1$ or $C_2$ or dihedral. The rider is
now clear.
\end{proof}

We end with a further comment on Corollary 6.6.
Embedding $R(\RR G)$ in $R(\CC G)$ via the tensor
product $\CC \otimes_\RR \dash$, we define an
elementary abelian $2$-group
$\ooO(G) = R(\RR G)/(R(\RR G) \cap I^+(\CC G))$.
Here, $I^+(\CC G)$ is the ideal of $R(\CC G)$
generated by the elements having the form
$\xi + \xi'$, where $\xi$ and $\xi'$ are Galois
conjugate $\CC G$-reps. The tom Dieck map
$\oodie_G^\RR$ gives rise to a map $\die_G :
\ooO(G) \rightarrow B(G)^\times$. Theorem 6.5
says that $\oodie_G$ is surjective. Corollary 6.6
says that, in fact, $\oodie_G$ is an isomorphism.
Let $B_0(G) = \Ker(\lin_G^\RR)$. Tornehave's
topological proof of Theorem 6.5 can be recast
in a purely algebraic way which involves maps
$\rip_G : B_0(G) \rightarrow \ooO(G)$ and
$\torn_G : B_0(G) \rightarrow B(G)^\times$
such that $\torn_G = \oodie_G \, \rip_G$. Such
commutative triangles can still be defined when
$G$ is replaced by an arbitrary group, although
$\oodie_G$ ceases to be an isomorphism in
general. The author intends, in a future paper,
to discuss some extensions and adaptations of
these results of Tornehave. The present paper
arose from that work.


\begin{thebibliography}{EMG}

\bibitem[1]{Asc86}
M.\ Aschbacher, ``Finite Group Theory'',
(Cambridge Univ.\ Press, 1986).

\bibitem[2]{Bou04}
S.\ Bouc, {\it The functor of rational
representations for $p$-groups}, Advances in
Math.\ {\bf 186}, 267-306 (2004).

\bibitem[3]{Bou05}
S. Bouc, {\it Biset functors and genetic sections for
$p$-groups}, J.\ Algebra, {\bf 284}, 179-202 (2005).

\bibitem[4]{Bou*}
S. Bouc, {\it The functor of units of Burnside rings
of $p$-groups}, (preprint).

\bibitem[5]{CR87}
C.\ W.\ Curtis, I.\ Reiner, ``Methods of Representation
Theory'', Vols.\ I and II, (Wiley, New York, 1981, 1987).

\bibitem[6]{Die79}
T.\ tom Dieck, ``Transformation Groups and
Representation Theory'' Lecture Notes in
Math.\ 766 (Springer, Berlin, 1979).

\bibitem[7]{Die87}
T.\ tom Dieck, ``Transformation Groups''
(De Gruyter, Berlin, 1987).

\bibitem[8]{Fei74}
B.\ Fein, {\it Minimal splitting fields for group
representations}, Pacific J.\ Math.\ {\bf 51},
427-431 (1974).

\bibitem[9]{HT99}
I.\ Hambleton, L.\ Taylor, {\it Rational permutation
modules for finite groups}, Math.\ Zeitschrift
{\bf 231}, 707-726 (1999).

\bibitem[10]{HTW90}
I.\ Hambleton, L.\ Taylor, B.\ Williams,
{\it Detection theorems for $K$-theory and
$L$-theory}, J.\ Pure and Applied Algebra,
{\bf 63}, 247-299 (1990).

\bibitem[11]{IY92}
Y.\ Iida, T.\ Yamanda, {\it Types of metacyclic
$2$-groups}, SUT J.\ Math.\ {\bf 28}, 237-262 (1992).

\bibitem[12]{Kro66}
K.\ Kronstein, {\it Characters and systems of
subgroups}. J. Reine Angew.\ Math.\ {\bf 224},
147-163 (1966).

\bibitem[13]{LM02}
C.\ R.\ Leedham-Green, S.\ McKay, ``The Structure
of Groups of Prime Power Order'', (Oxford Univ.\
Press, 2002).

\bibitem[14]{Roq58}
P.\ Roquette, {\it Realiseirung von Darstellungen
endlicher nilpotenter Gruppen}, Arch.\ Math.
{\bf 9}, 241-250 (1958).

\bibitem[15]{Tor84}
J.\ Tornehave, {\it The unit theorem for the
Burnside ring of a $2$-group}, Aarhaus Universitet,
Preprint Series 1983/84, No.\ 41 (May 1984).

\bibitem[16]{Wit52}
W.\ Witt, {\it Die algebraische Struktur des
Gruppenrings einer endlichen Gruppe \'{u}ber
einem Zahlenk\"{o}rper}, J.\ Reine Angew.\ Math.\
{\bf 190}, 231-245 (1952).

\bibitem[17]{Yal*}
E.\ Yal\c{c}{\i}n, {\it An induction theorem for
the unit groups of Burnside rings of $2$-groups},
J.\ Algebra {\bf 289}, 105-127 (2005).

\bibitem[18]{YI93}
T.\ Yamada, Y.\ Iida, {\it Rational representations,
types and extensions of $2$-groups}, 69-82 in
``Proc.\ 26th Symp.\ Ring Theory (Tokyo 1993)'',
(Okuyama Univ., 1993).

\bibitem[19]{Yos90}
T.\ Yoshida, {\it On the unit groups of Burnside
rings}, J.\ Math.\ Soc.\ Japan {\bf 42},
31-64 (1990).

\end{thebibliography}
\end{document}